\documentclass[microtype]{gtpart}     % Basic GT/GTM/AGT style
\usepackage{pinlabel}  %%% the recommended graphics+labelling package
\usepackage{graphicx}  %%% the recommended graphics package

\usepackage{amsfonts,mathtools,mathrsfs}
\usepackage{tikz-cd}
\usepackage{cite}

\title{Action of the Mazur Pattern up to topological concordance}

\author{Alex Manchester}
\givenname{Alex}
\surname{Manchester}
\address{Rice Univ Math Dept---MS 136\\ PO Box 1892\\ Houston TX 77005}
\email{Alex.Manchester@rice.edu}
\urladdr{https://sites.google.com/view/alexmanchester/home}


\arxivreference{2212.13640}

\newtheorem{theorem}{Theorem}[section]   % Standard theorem environment
\newtheorem{lemma}[theorem]{Lemma}       % Lemma environment with numbering 
\newtheorem{prop}[theorem]{Proposition}
\newtheorem{cor}[theorem]{Corollary}
\theoremstyle{definition}
\newtheorem{definition}[theorem]{Definition}
\theoremstyle{remark}
\newtheorem{question}[theorem]{Question}
\newtheorem{remark}[theorem]{Remark}
\newtheorem{conj}[theorem]{Conjecture}
\newcommand{\N}{{\mathbb{N}}}

\newcommand{\cA}{{\mathcal{A}}}
\newcommand{\cC}{{\mathcal{C}}}
\newcommand{\cF}{{\mathcal{F}}}

\newcommand{\Bl}{{\mathrm{Bl}}}
\newcommand{\Ch}{{\mathrm{Ch}}}
\newcommand{\Hom}{{\mathrm{Hom}}}
\newcommand{\ab}{{\mathrm{ab}}}
\newcommand{\id}{{\mathrm{id}}}

\newcommand{\lk}{{\mathrm{lk}}}
\newcommand{\metab}{{\mathrm{mab}}}
\newcommand{\sm}{{\mathrm{sm}}}
\newcommand{\ttop}{{\mathrm{top}}}

\begin{document}

\begin{abstract}    % type your abstract below
In the `80s, Freedman showed that the Whitehead doubling operator acts trivally up to topological concordance. On the other hand, Akbulut showed that the Whitehead doubling operator acts nontrivially up to smooth concordance. The Mazur pattern is a natural candidate for a satellite operator which acts by the identity up to topological concordance but not up to smooth concordance. Recently there has been a resurgence of study of the action of the Mazur pattern up to concordance in the smooth and topological categories. Examples showing that the Mazur pattern does not act by the identity up to smooth concordance have been given by Cochran--Franklin--Hedden--Horn and Collins. In this paper, we give evidence that the Mazur pattern acts by the identity up to topological concordance.

In particular, we show that two satellite operators $P_{K_0,\eta_0}$ and $P_{K_1,\eta_1}$ with $\eta_0$ and $\eta_1$ freely homotopic have the same action on the topological concordance group modulo the subgroup of $(1)$-solvable knots, which gives evidence that they act in the same way up to topological concordance. In particular, the Mazur pattern and the identity operator are related in this way, and so this is evidence for the topological side of the analogy to the Whitehead doubling operator. We give additional evidence that they have the same action on the full topological concordance group by showing that up to topological concordance they cannot be distinguished by Casson-Gordon invariants or metabelian $\rho$-invariants.
\end{abstract}

\maketitle

%%%%%%%%%%%%%%%%%%%%   Start of main body of article

\section{Introduction}
\label{Intro}

The Whitehead doubling operator, shown in Figure \ref{fig:SatelliteOperator}, is a fundamental example of a satellite operator. The strands going through the box should be tied into the knot $J$. A general satellite operator acts on a knot $J$ by grabbing a collection of strands and tying them into $J$. We will give a precise definition in Section \ref{Background}.

\begin{figure}
	\centering
	\includegraphics[width=0.5\linewidth]{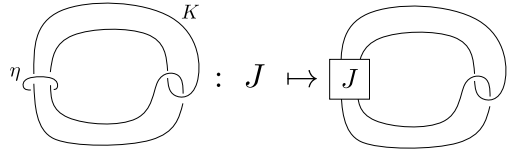}
	\caption{The Whitehead doubling operator}
	\label{fig:SatelliteOperator}
\end{figure}

Another important satellite operator is the {\it Mazur pattern}, which is shown on the left in Figure \ref{fig:HomRelSatOps}, along with the identity operator on the right, which takes every knot to itself. These two operators are closely related in a way which we will make precise later in Definition \ref{HomotopicallyRelated}, and for now we will simply note that $\lk(K_0,\eta_0)=\lk(K_1,\eta_1)$ and that both operators take the unknot to itself. The main theorem of this paper, Theorem \ref{1solvable}, essentially says that satellite operators which are related in a certain way induce the same function on knots up to some equivalence relation.

\begin{figure}
	\centering
	\includegraphics[width=0.5\linewidth]{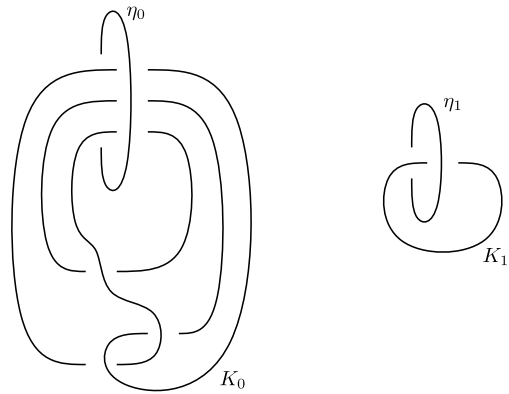}
	\caption{The Mazur pattern and the identity operator are homotopically related.}
	\label{fig:HomRelSatOps}
\end{figure}

All knots are assumed to be oriented, and given a knot $K$, the knot $-K$ is the mirror image of $K$ with the reverse orientation. Unless otherwise specified, manifolds and submanifolds (in particular slice disks and concordances) are smooth, though we will often use the words ``smooth" and ``smoothly" for emphasis.

Satellite operators give well-defined maps on the concordance groups $\cC_\ttop$ and $\cC_\sm$ which are of significant interest. For example, their kernels and images have been studied, and of particular note is the result of Cochran--Davis--Ray that satellite operators with strong winding number $\pm 1$ are injective up to topological concordance (see\cite{cochranInjectivitySatelliteOperators2014}, especially Theorem 5.1, as well as work of Hedden--Pinz\'{o}n-Caicedo
\cite{heddenSatellitesInfiniteRank2021} and Levine
\cite{levineNONSURJECTIVESATELLITEOPERATORS2016}). Also, satellite operators have been used to give evidence for a ``fractal nature'' of the concordance groups (see Cochran--Harvey
\cite{cochranGeometryKnotConcordance2018}, Cochran--Harvey--Leidy
\cite{cochranPrimaryDecompositionFractal2011}, Cochran--Harvey--Powell
\cite{cochranGropeMetricsKnot2017}, and Ray
\cite{raySatelliteOperatorsDistinct2015}). Furthermore, they are typically not homomorphisms, and it is conjectured that only three quite trivial satellite operators are homomorphisms (see Lidman--Miller--Pinz\'{o}n-Caicedo \cite{lidmanLinkingNumberObstructions2022}, and Miller
\cite{millerHomomorphismObstructionsSatellite2019}). Instead, satellite operators can be interpreted as group actions (see Davis--Ray
\cite{davisSatelliteOperatorsGroup2016}). Satellite operators have also been used to give evidence for and against the slice-ribbon conjecture (see Gompf--Miyazaki
\cite{gompfWelldisguisedRibbonKnots1995}, Miller--Piccirillo
\cite{millerKnotTracesConcordance2018}, and Yasui
\cite{yasuiCorksExotic4manifolds2017}).

The following theorem is a special case of Theorem 1.5 in \cite{cochranNewConstructionsSlice2007} of Cochran--Friedl--Teichner, and is motivation for this paper, particularly the main result Theorem \ref{1solvable}.

\begin{theorem}[see Theorem 1.5 in \cite{cochranNewConstructionsSlice2007}]
\label{COTThm}
If $K$ is topologically (or smoothly) slice, and $\eta$ is nullhomotopic in some topological slice disk complement for $K$, then $P(J)$ is topologically slice for any $J$ (in other words, $P:\cC_\ttop\to\cC_\ttop$ is the zero map).
\end{theorem}

Note that the Whitehead doubling operator shown in Figure \ref{fig:SatelliteOperator} satisfies the condition of Theorem \ref{COTThm} (since $K$ is the unknot and so has the trivial slice disk whose complement has fundamental group $\Z$, and $\eta$ has linking number 0 with $K$ and so must be nullhomotopic). So, this theorem gives a proof that all Whitehead doubles are topologically slice. The fact that all Whitehead doubles are topologically slice was first proved by Freedman in \cite{freedmanNewTechniqueLink1985}, and Theorem \ref{COTThm} is a large generalization of Freedman's result. It is also worth noting that all Whitehead doubles have Alexander polynomial 1, and Freedman--Quinn showed that all knots with Alexander polynomial 1 are topologically slice in Theorem 11.7B of \cite{freedmanTopology4manifolds1990}.

We think of the following definition as being a relative version of the condition in Theorem \ref{COTThm}.

\begin{definition}
\label{HomotopicallyRelated}
Two satellite operators $P_0$ and $P_1$ defined by the links $K_0\cup\eta_0$ and $K_1\cup\eta_1$, respectively, are \emph{homotopically related} if $K_0$ and $K_1$ are concordant and $\eta_0$ and $\eta_1$ cobound an immersed annulus in the complement of the concordance (that is, if $\eta_0\times\{0\}$ and $\eta_1\times\{1\}$ are freely homotopic in the concordance complement). Note that this definition makes sense in both the smooth and topological settings; however, we will primarily make use of it in the smooth setting.
\end{definition}

Satellite operators which satisfy the condition of Theorem \ref{COTThm} are exactly those which are topologically homotopically related to the satellite operator given by the unlink $K\sqcup\eta$.

Two (smoothly) homotopically related satellite operators which do not satisfy the condition of Theorem \ref{COTThm} are the aforementioned Mazur pattern and identity operator, shown in Figure \ref{fig:HomRelSatOps}. These two can be seen to be homotopically related since $K_0$ is the unknot, and so its complement has fundamental group $\Z$, as does the complement of the concordance from the unknot to itself. And so, $\eta_0$ is freely homotopic to a meridian since it has linking number 1 with $K_0$, which gives the immersed annulus.

Theorem \ref{COTThm} motivates the following question, asked by Ray at a 2019 AIM workshop \cite{SmoothConcordanceClasses}:

\begin{question}
\label{MainQuestion}
For homotopically related satellite operators $P_0$ and $P_1$, is $P_0(J)$ topologically concordant to $P_1(J)$ for each $J$? In other words, is $P_0(J)\#-P_1(J)$ topologically slice?
\end{question}

Note that if the answer to Question \ref{MainQuestion} is ``yes,'' this would imply that the actions of satellite operators $P$ which take the unknot to itself (for example, the Whitehead doubling operator, the Mazur pattern, and the identity operator) are completely determined up to topological concordance by their winding number via an argument similar to the one above.

In this paper, we give some evidence that the answer to this question is ``yes,'' in particular the following theorem. The condition of (1)-solvability lies between algebraic sliceness and sliceness, and will be defined later.
\begin{theorem}
\label{1solvable}
Given (smoothly) homotopically related satellite operators $P_\epsilon$ for $\epsilon\in\{0,1\}$, for any $J$ the knot $P_0(J)\#-P_1(J)$ is (1)-solvable. In other words, $P_0 \equiv P_1$ on $\cC/\cF_{(1)}$, where $\cF_{(1)}$ is the subgroup of (1)-solvable knots.
\end{theorem}

In fact, at the end of Section \ref{Construction}, we will show that $P_0(J)\#-P_1(J)$ is {\it homotopy ribbon (1)-solvable} when $K_0\#-K_1$ is homotopy ribbon, which is slightly stronger than being (1)-solvable.

%The fact that $P_0(J)\#-P_1(J)$ is algebraically slice (or equivalently (0.5)-solvable) already follows directly from results of Litherland (see \cite{litherlandCobordismSatelliteKnots1984}, \cite{litherlandSymmetriesTwistSpunKnots1985}) which will be discussed in Section \ref{CGinvariants}.

In the particular case of the Mazur pattern, denoted by $Q$, the 0-surgeries $M_J$ and $M_{Q(J)}$ are (smoothly) homology cobordant rel meridian. In fact this is true for any satellite operator with winding number 1 which takes the unknot to itself, as shown by Cochran--Franklin--Hedden--Horn in Corollary 2.2 of \cite{cochranKnotConcordanceHomology2013a}. It is an open question whether any two knots whose 0-surgeries are smoothly or topologically homology cobordant rel meridian are topologically concordant. %(Any two topologically concordant knots have 0-surgeries which are homology cobordant rel meridian, which can be seen by doing 0-surgery ``$\times I$'' along the concordance.)

In the smooth setting, Theorem \ref{COTThm} is false, and the answer to Question \ref{MainQuestion} is ``no.'' In unpublished work in 1983, Akbulut used gauge theory to show that the Whitehead double of the right-handed trefoil is not smoothly slice. Gauge theoretic techniques are very sensitive to matters of handedness, and it is still unknown whether the Whitehead double of the left-handed trefoil is smoothly slice. The question of smooth sliceness of the Whitehead double of the left-handed trefoil is one of the most important questions in low-dimensional topology. Furthermore, one can use Heegaard-Floer theory to show that for many knots $J$ (in particular the figure-eight knot), the knots $J$ and $Q(J)$ are not necessarily smoothly concordant (see Theorem 3.1 of Cochran--Franklin--Hedden--Horn \cite{cochranKnotConcordanceHomology2013a}, and Proposition 1.2 of Collins  \cite{collinsHomologyCobordismSmooth2022}).

In Section \ref{Background}, we will give some basic definitions and background. In Section \ref{Construction}, we will construct a (1)-solution for $P_0(J)\#-P_1(J)$ (which is moreover a homotopy ribbon (1)-solution when $K_0\#K_1$ is homotopy ribbon). In Section \ref{SlicenessConditions}, we will discuss some additional conditions which would guarantee that these knots are topologically slice. In Sections \ref{CGinvariants} and \ref{MetaRhoInvs}, we will show that $P_0(J)$ and $P_1(J)$ cannot be distinguished up to topological concordance using certain obstructions coming from Casson-Gordon invariants and metabelian $\rho$ invariants, respectively, which gives further evidence that these two knots are topologically concordant.

\subsection*{Acknowledgements}

I would like to thank Shelly Harvey, Miriam Kuzbary, and Allison Miller for their many valuable conversations and useful suggestions. I was partially supported by the RTG award NSF DMS-1745670.

\section{Background}
\label{Background}

We begin with a precise definition of satellite operators:

\begin{definition}
A {\it satellite operator} $P=P_{K,\eta}$ is given by a two component link $K\cup\eta$, where $\eta$ is an unknot, and acts on the set of knots in $S^3$ as follows: Notice that the exterior of $\eta$ is a solid torus $E_\eta$ which contains the knot $K$ and which has as a preferred longitude a meridian $\mu_\eta$ of $\eta$. The result of applying $P$ to $J$ is the knot given by deleting a tubular neighborhood of $J \subset S^3$ and gluing in $E_\eta$ in such a way as to identify the meridian of $\eta$ with the preferred longitude of $J$ and vice versa (which guarantees that the ambient manifold is still $S^3$). Then $P(J)$ is the image of $K$ in the resulting $S^3$. In other words, a satellite operator acts by gluing together the exterior $E_\eta$ of $\eta$ and the exterior of $J$ along their torus boundary via the homeomorphism $T^2 \to T^2$ which identifies the meridian of $\eta$ with the preferred longitude of $J$ and vice versa, and while doing this, keeps track of the knot $K \subset E_\eta$.
\end{definition}

Notice that $P(U)=K$ since when $J=U$, the exterior of $J$ is just a solid torus, which, when glued to $E_\eta$ simply replaces the solid torus neighborhood of $\eta$ in the trivial way.

	A knot in $S^3$ is called {\it smoothly/topologically slice} if it is the boundary of a smooth/locally flat properly embedded disk in $B^4$. Two knots $K$ and $J$ are {\it smoothly/topologically concordant} if there is a smooth/locally flat properly embedded annulus in $S^3 \times I$ with boundary $(K\times\{0\})\sqcup(-J\times\{1\})$. The set of equivalence classes of knots up to smooth/topological concordance is denoted $\cC_{\sm}$ or $\cC_{\ttop}$, respectively.

The sets $\cC_{\sm}$ and $\cC_{\ttop}$ both form groups under the connect sum operation $\#$, with $[K]^{-1}=[-K]$. We will abuse notation and write $K$ for its concordance class. The fact that $-K=K^{-1}$ can be seen by taking the pair $(S^3 \times I, K \times I)$, and removing a tubular neighborhood of an arc running from a point on $K\times\{0\}$ to a point on $-K\times\{1\}$ (note that taking $S^3$ with the standard orientation means that the boundary on the ``1'' side is $(S^3,-K)$), after which the remainder of $S^3 \times I$ is a 3-ball and the remainder of the annulus $K \times I$ is a slice disk for $K\#-K \subset S^3 \# S^3 = S^3$. (This construction works in either the smooth or the topological setting.)

Note that $K$ and $J$ are (smoothly or topologically) concordant if and only if the connect sum $K\#-J$ is (smoothly or topologically) slice, by an argument similar to that in the previous paragraph showing that $[K]^{-1}=[-K]$.

Satellite operators descend to maps $\cC_\sm\to\cC_\sm$ and $\cC_\ttop\to\cC_\ttop$ (but, as noted before, are typically not homomorphisms, and can instead be thought of as group actions).

We will need the following well-known characterization of topological sliceness.

\begin{prop}[see e.g. Prop 2.1 of \cite{cochranNewConstructionsSlice2007}]
\label{TopSliceCondition}
A knot $K$ is topologically slice if and only if its 0-surgery $M_K$ bounds a topological 4-manifold $W$ such that
	\begin{enumerate}
	\item $\pi_1(W)$ is normally generated by the meridian $\mu$ of $K$,
	\item $H_1(W)\cong\Z$ (in other words, the inclusion map $H_1(M_K) \to H_1(W)$ is an isomorphism), and
	\item $H_2(W)=0$.
	\end{enumerate}
\end{prop}

It is often difficult to check whether $M_K$ bounds such a topological 4-manifold, but sometimes it is easier to show that $M_K$ bounds a smooth 4-manifold satisfying the following weaker condition, due to Cochran-Orr-Teichner in \cite{cochranKnotConcordanceWhitney2003}. Recall that the derived series of a group $G$ is defined inductively as $G^{(0)}:=G$ and $G^{(n+1)}:=[G^{(n)}:G^{(n)}]$, that is, each term is the commutator subgroup of the previous term.

\begin{definition}
\label{Solvability}
A knot $K$ is called \emph{$(n)$-solvable} if $M_K$ bounds a smooth 4-manifold $W$ such that
\begin{enumerate}
	\item the inclusion map $H_1(M_K) \to H_1(W)$ is an isomorphism,
	\item $H_2(W)$ is freely generated by embedded surfaces $\{S_i, T_i\}$ with trivial normal bundle such that $S_i$ and $T_i$ intersect once transversely, and otherwise the surfaces are disjoint, and
	\item the inclusion maps $\pi_1(S_i),\pi(T_i)\to\pi(W)$ land in $\pi_1(W)^{(n)}$
\end{enumerate}
If, additionally, the inclusion maps $\pi_1(S_i)\to\pi_1(W)$ land in $\pi_1(W)^{(n+1)}$, then $K$ is said to be \it{$(n.5)$-solvable}. If $K$ is $(h)$-solvable for some $h\in\frac{1}{2}\N$, the manifold $W$ is called an \emph{$(h)$-solution} for $K$.
\end{definition}

\begin{remark}
	Notice that smooth concordance preserves $(h)$-solvability since if $K$ and $J$ are concordant and $J$ is $(h)$-solvable, we can construct an $(h)$-solution for $K$ in the following way: Do 0-surgery $\times I$ along a concordance from $K$ to $J$, and glue an $(h)$-solution for $J$ to the $M_J$ boundary component. By a Seifert-van Kampen and Mayer-Vietoris argument, this gives an $(h)$-solution for $K$ with the surfaces representing generators for the second homology simply the images of the surfaces representing the second homology in the $(h)$-solution for $J$.
	
	Also, if $K$ and $J$ are both $(h)$-solvable, with $(h)$-solutions $W_K$ and $W_J$, respectively, one can construct an $(h)$-solution $W$ for $K\#J$ by gluing together $W_K$ and $W_J$ along the solid tori in their boundaries given by tubular neighborhoods of meridians of $K$ and $J$ respectively, identifying the meridians and preferred longitudes in the torus boundaries. The boundary of $W$ is then $M_{K\#J}$, which can by seen by viewing $K\#\cdot$ as a satellite operator given by $K\cup\mu_K$: the tubular neighborhoods of meridians of $K$ and $J$ are now buried in the interior of $W$, and the boundary is given by gluing together the exteriors of $\mu_K$ and $\mu_J$ (note that $\mu_J$ can be freely homotoped to be the core of the surgery solid torus in $M_J$, and so the exterior of $\mu_J$ is just the exterior of $J$). Of course, we could just have well have interchanged the roles of $K$ and $J$. Again by Seifert-van Kampen and Mayer-Vietoris, this $W$ is an $(h)$-solution for $K\#-J$.
	
	Lastly, if $K$ is $(h)$-solvable then so is $-K$ by taking the mirror image of an $(h)$-solution for $K$. So, the $(h)$-solvable filtration is a filtration by subgroups.
\end{remark}

In the previous remark, we used the fact that the 0-surgery on a connect sum $K\#J$ is homeomorphic to the union of the exteriors of $K$ and $J$, glued along their torus boundary in such a way as to identify the meridians and longitudes. We will use this fact at a few points throughout this paper.

\begin{definition}
	The set of $(h)$-solvable knots up to smooth concordance is denoted $\cF_{(h)}$.
\end{definition}

This defines a filtration of the concordance group $\cC_\sm$:
	$$\cdots\subseteq\cF_{(1.5)}\subseteq\cF_{(1)}\subseteq\cF_{(0.5)}\subseteq\cF_{(0)}\subseteq\cC_\sm$$
This filtration was first defined by Cochran-Orr-Teichner in \cite{cochranKnotConcordanceWhitney2003}. Note that if the surfaces $S_i,T_i$ were spheres, condition (3) of Definition \ref{Solvability} would be trivially satisfied for all $h$. In this case, we could do surgery on one sphere from each pair, replacing a tubular neighborhood $S^2 \times D^2$ with $B^3 \times S^1$ (which have the same boundary), in order to get a manifold which satisfies the conditions of Prop \ref{TopSliceCondition}. All topologically slice knots are $(h)$-solvable for every $h\in\frac{1}{2}\N$ as noted in Section 1 of \cite{cochranKnotConcordanceHigherorder2009}, and it is conjectured that $\bigcap_{h\in\frac{1}{2}\N}\cF_{(h)}$ contains precisely the topologically slice knots. It is known that for all $n\in\N$, the quotient $\cF_{(n)}/\cF_{(n.5)}$ has infinite rank (see \cite{cochranKnotConcordanceHigherorder2009}). However, it is unknown whether $\cF_{(n.5)}/\cF_{(n+1)}$ is even nontrivial for any $n$.

\section{Construction of a (1)-solution}
\label{Construction}

In this section we will prove Theorem \ref{1solvable} by constructing a (1)-solution for $P_0(J)\#-P_1(J)$.

Fix the following notation: Let $P_\epsilon$ be homotopically related satellite operators coming from the links $K_\epsilon\cup\eta_\epsilon$, $\epsilon\in\{0,1\}$. Let $C$ be the (smooth) concordance between $K_0$ and $K_1$, and let $W = (S^3 \times I) \setminus \nu(C)$. Note that $W$ is a manifold satisfying the conditions of Proposition \ref{TopSliceCondition} for $K_0\#-K_1$. Let $J$ be a knot, and let $P_\epsilon(J)$ be the image of $J$ under the satellite operator $P_\epsilon$. Let $\mu$ be a fixed meridian of $J$.

 Let $N$ be the manifold obtained by gluing two copies of $M_J \times I$ to $W$ in the following way (see Figure \ref{fig:NConstruction} for a schematic picture.): recall that $M_J=(S^3 \setminus J)\cup_f(S^1 \times D^2)$ where $f : \partial(S^1 \times D^2) \to \partial (S^3 \setminus J)$ identifies the meridian of $S^1 \times D^2$ with the preferred longitude of $J$ and the longitude of $S^1 \times D^2$ with the meridian of $J$. Then $N$ is obtained by gluing two copies of $M_J \times I$, which we will denote by $(M_J \times I)_\epsilon$ for $\epsilon\in\{0,1\}$, to $W$, by identifying the surgery solid torus $S^1 \times D^2 \subset (M_J \times \{0\})_\epsilon$ with $\nu(\eta_\epsilon)$ in such a way as to identify the longitude of the surgery solid torus (that is, a meridian of $J$) with the 0-framed longitude of $\eta$. This gluing ``buries'' the two solid tori that get glued together inside the interior of the new 4-manifold. And, the effect of these identifications on the boundary amounts to gluing the exteriors of $J$ and $\eta$ together via the homeomorphism $T^2 \to T^2$ which interchanges meridians and longitudes, so that $\partial N = M_{P_0(J)\#-P_1(J)}\sqcup-M_J\sqcup-M_J$. Denote the two copies of $\mu\times\{1\}$ by $\mu_\epsilon$, and let $A$ be the immersed annulus in $W$ joining the $\eta_\epsilon$, extended through $(M_J \times I)_\epsilon$ by crossing with $I$, so that its boundary components are the two $\mu_\epsilon$. Finally, let $\hat{\Sigma}_\epsilon$ denote Seifert surfaces for $J$ capped off by longitudinal disks in the copies of $(M_J \times \{1\})_\epsilon$.

\begin{figure}
	\begin{center}
	\includegraphics[width=0.3\linewidth]{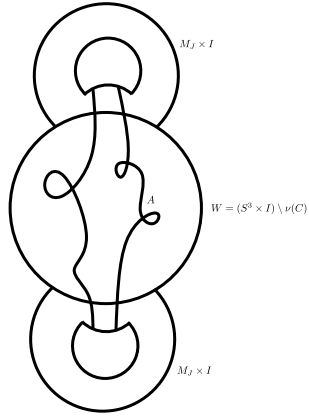}
	\end{center}
	\caption{Construction of $N$: Two copies of $M_J \times I$ are glued to $W=(S^3 \times I)\setminus\nu(C)$ along solid torus portions of their respective boundaries. This ``buries'' the solid tori in the interior of $N$, and the new boundary is $M_{P_0(J)\#-P_1(J)}\sqcup-M_J\sqcup-M_J$.}
	\label{fig:NConstruction}
\end{figure}

We will now construct a (1)-solution for $P_0(J)\#-P_1(J)$. This will involve cutting out a neighborhood of $A$, giving a manifold with two boundary components, one of which is $M_{P_0(J)\#-P_1(J)}$. We will then find an appropriate manifold with which to ``cap off'' the other boundary component, which will give us our (1)-solution.

This will involve a series of lemmas keeping track of the fundamental group and homology of the intermediate steps in the construction. Throughout, we will only care about normal generating sets for fundamental groups and subgroups, and so we will only care about elements of the normal generating sets up to conjugacy, and in turn will only care about their representatives up to free homotopy. Given group elements $g_1,...,g_n \in G$, we denote the subgroup they normally generate by $\langle\langle g_1,...,g_n \rangle\rangle$.

\begin{lemma}
\begin{enumerate}
\item $\pi_1(N)$ is normally generated by the meridian of $P_0(J)$ (or equivalently of $P_1(J)$, $K_0$, or $K_1$ since these meridians are all freely homotopic inside $N$).
\item $H_1(S^3 \setminus \nu(P_\epsilon(J))) \to H_1(N)$ is an isomorphism. That is, $H_1(N)\cong\Z$ and is generated by the meridian of $P_0(J)$, or equivalently of $P_1(J)$, $K_0$, or $K_1$.
\item $H_2(M_J) \oplus H_2(M_J) \to H_2(N)$ is an isomorphism. That is, $H_2(N)\cong\Z^2$ and is generated by the $\hat{\Sigma}_\epsilon$.
\end{enumerate}
\end{lemma}

\begin{proof}
(1) $\pi_1(W)$ is normally generated by the meridian of $K_0$ (or equivalently of $K_1$ since these are freely homotopic and so conjugate in $\pi_1(W)$). In $N$, a meridian of $K_\epsilon$ is freely homotopic to a meridian of $P_\epsilon(J)$. Moreover, $\pi_1((M_J \times I)_\epsilon)$ is normally generated by $\mu_\epsilon$, which is freely homotopic to the curve $\mu\times\{0\}\subset(M_J \times I)_\epsilon$ identified with $\eta_\epsilon$ under the gluing used to construct $N$, and so is in the subgroup normally generated by a meridian of $K_0$. Therefore, by Seifert-van Kampen, $\pi_1(N)=\pi_1(W)*_{\langle\langle\eta_0\mu_0^{-1}\rangle\rangle}\pi_1(M_J \times I)*_{\langle\langle\eta_1\mu_1^{-1}\rangle\rangle}\pi_1(M_J \times I)$ is normally generated by a meridian of $K_0$, or equivalently of $K_1$, $P_0(J)$, or $P_1(J)$ since these are freely homotopic in $N$.

\setlength\mathsurround{0pt}

(2 and 3) If we attach one copy of $M_J \times I$, say the one glued to $\eta_0$, we have the following Mayer-Vietoris sequence:
	$$\begin{tikzcd}
		0 \arrow[r]
			& 0 \oplus H_2 (M_J \times I) \arrow[r] \arrow[d, phantom, ""{coordinate, name=Z}]
			& H_2(W \cup (M_J \times I))  \arrow[dll,
											rounded corners,
											to path={ -- ([xshift=2ex]\tikztostart.east)
											|- (Z) [near end]\tikztonodes
											-| ([xshift=-2ex]\tikztotarget.west)
											-- (\tikztotarget)}] \\
		H_1(\nu(\eta_0)) \arrow[r]
			& H_1(W) \oplus H_1(M_J \times I) \arrow[r]
			& H_1(W \cup (M_J \times I)) \arrow[r]
			& 0
	\end{tikzcd}.$$

\setlength\mathsurround{0.8pt}

Since $\eta_0$ is identified with $\mu\times\{0\}\subset(M_J\times\{0\})_0$, which is freely homotopic to $\mu_0$, the map $H_1(\nu(\eta_0)) \to H_1(M_J \times I)$ is an isomorphism. This implies that the maps $H_1(W) \to H_1(W \cup (M_J \times I))$ and $H_2(M_J \times I) \to H_2(W \cup (M_J \times I))$ are isomorphisms.\footnote{To see that the map $H_1(W) \to H_1(W \cup (M_J \times I))$ is an isomorphism, choose the basis for $H_1(W) \oplus H_1(M_J \times I)$ consisting of the image of a generator of $H_1(\nu(\eta_0))$ (which is $\{\text{some element of }H_1(W)\} \oplus \{\text{a generator of }H_1(M_J \times I)\}$), and $\{\text{a generator of }H_1(W)\}\oplus\{0\}$. This is a basis since $H_1(\nu(\eta_0)) \to H_1(M_J \times I)$ is an isomorphism. Then, one can see that $H_1(W \cup (M_J \times I))$ is generated by the image of a generator of $H_1(W)$. We apply a similar argument later in this proof, and also in the proof of Lemma \ref{N'Homology}.} And noting that $H_1(S^3 \setminus \nu(P_\epsilon(J))) \cong H_1(S^3 \setminus \nu(K_\epsilon)) \to H_1(W)$ is an isomorphism, we see that $H_1(S^3 \setminus \nu(P_\epsilon(J))) \to H_1(W)$ is an isomorphism.

Similarly, attaching another copy of $M_J \times I$ to the other side yields

\setlength\mathsurround{0pt}

	$$\begin{tikzcd}
		0 \arrow[r]
			& H_2(W \cup (M_J \times I)) \oplus H_2 (M_J \times I) \arrow[r] \arrow[d, phantom, ""{coordinate, name=Z}]
			& H_2(N)  \arrow[dll,
									rounded corners,
									to path={ -- ([xshift=2ex]\tikztostart.east)
									|- (Z) [near end]\tikztonodes
									-| ([xshift=-2ex]\tikztotarget.west)
									-- (\tikztotarget)}] \\
		H_1(\nu(\eta_1)) \arrow[r]
			& H_1(W \cup (M_J \times I)) \oplus H_1(M_J \times I) \arrow[r]
			& H_1(N) \arrow[r]
			& 0
	\end{tikzcd}.$$
	
\setlength\mathsurround{0.8pt}
	
Similarly, since $H_1(\nu(\eta_1)) \to H_1(M_J \times I)$ is an isomorphism, we have that $H_1(W \cup (M_J \times I)) \to H_1(N)$ and $H_2(W \cup (M_J \times I)) \oplus H_2(M_J \times I) \to H_2(N)$ are isomorphisms. We can compose these with the isomorphisms from the first Mayer-Vietoris sequence to obtain isomorphisms $H_1(S^3 \setminus \nu(P_\epsilon(J))) \to H_1(N)$ and $H_2(M_J) \oplus H_2(M_J) \to H_2(N)$.
\end{proof}

From now on, let $\lambda$ denote the normal generator of $\pi_1(N)$ from the previous lemma. Recall that $\mu_\epsilon$ are the meridians of $J$ in $M_J\times\{1\}\subset(M_J \times I)_\epsilon$, and $A$ is the immersed annulus cobounded by the $\eta_\epsilon$, extended through the $(M_J \times I)_\epsilon$ to the boundary $-M_J$s via the product structure so that it is cobounded by $\mu_\epsilon$.

\begin{lemma}
\begin{enumerate}
\item $\pi_1(N\setminus\nu(A))$ is normally generated by $\lambda$ and a meridian of $A$.
\item $H_1(N \setminus\nu(A)) \to H_1(N)$ is an isomorphism.
\item $H_2(N \setminus\nu(A)) \cong \Z \oplus G$ (where $\hat{\Sigma}_0-\hat{\Sigma}_1$ generates the $\Z$  summand, and the $G$ summand has one generator for each self-intersection of $A$, with possibly some relations)
\end{enumerate}
\end{lemma}

\begin{proof}
(1) follows from Seifert-van Kampen.

Define $\partial'\nu(A)$ by the following decomposition: $\partial\nu(A) = \nu(\partial A) \sqcup \partial'\nu(A) = \nu(\eta_0)\sqcup\nu(\eta_1)\sqcup\partial'\nu(A)$. Then
\begin{align*}
	H_n(N,N\setminus\nu(A))
		& = H_n(\nu A,\partial'\nu(A))
			& \text{by excision}\\
		& = H^{4-n}(\nu A , \nu(\partial A))
			& \text{by Poincar\'{e}-Lefschetz duality}\\
		& = H^{4-n}(A, \partial A)
			& \text{by a deformation retract}\\
		& = \begin{cases}
			\Z			& n=2\\
			\Z^{c+1}	& n=3\\
			0			& \text{else}
			\end{cases}
			& \begin{multlined}
				\text{where $c$ is the number of}\\
				\text{self-intersections of $A$}
			\end{multlined}
\end{align*}
where the final step comes from finding a CW structure on an annulus with $c$ pairs of points identified and computing the cohomology directly from the cochain complex. The group $H^1(A,\partial(A))$ has one generator corresponding to each self-intersection of $A$, plus one generator which is represented by an arc going from one boundary component to the other. Following through the isomorphisms, we see that the corresponding element of $H_3(N,N \setminus \nu(A))$ is represented by a tubular neighborhood of a loop going around $A$ (this is a Poincar\'{e}-Lefschetz dual of the element in $H^1(A,\partial A)$ since it intersects the arc representing it exactly once), with boundary a torus in $\partial'\nu(A) \subset N\setminus\nu(A)$, which can be taken without loss of generality to be isotopic to the boundary of a tubular neighborhood of either of the $\mu_\epsilon$ by choosing the loop to be a push off of one of the boundary components. Also, the generator of $H_2(N,N\setminus\nu(A))$ is given by a Poincar\'{e}-Lefschetz dual of $A$, since the generator of $H^2(A, \partial A)$ is represented by $A$.

These homology groups then fit in to the long exact sequence for the pair $(N,N\setminus\nu(A))$:

\setlength\mathsurround{0pt}

$$\begin{tikzcd}
		H_3(N, N \setminus \nu(A)) \arrow[r]
			& H_2(N \setminus \nu(A)) \arrow[r] \arrow[d, phantom, ""{coordinate, name=Z}]
			& H_2(N)  \arrow[dll,
									rounded corners,
									to path={ -- ([xshift=2ex]\tikztostart.east) 
									|- (Z) [near end]\tikztonodes
									-| ([xshift=-2ex]\tikztotarget.west)
									-- (\tikztotarget)},
									swap, "\pi"] \\
		H_2(N, N \setminus \nu(A)) \arrow[r]
			& H_1(N \setminus \nu(A)) \arrow[r]
			& H_1(N) \arrow[r]
			& 0
	\end{tikzcd}.$$

\setlength\mathsurround{0.8pt}

Note that the map $\pi : H_2(N) \to H_2(N, N \setminus \nu(A))$ is given by the algebraic intersection number with $A$ since the generator of $H_2(N,N \setminus \nu(A))$ is the Poincar\'{e}-Lefschetz dual of $A$. Since the boundary components of $A$ are the meridians $\mu_\epsilon$ of $J$ in $(M_J \times I)_\epsilon$, and the capped off Seifert surfaces $\hat{\Sigma}_\epsilon$ intersect these meridians once, we see that the images of the classes $\hat{\Sigma}_\epsilon \in H_2(N)$ under $\pi$ are $\pm 1 \in \Z \cong H_2(N, N \setminus \nu(A))$. This shows that $\pi$ is surjective and has kernel $\langle\hat{\Sigma}_0-\hat{\Sigma}_1\rangle$. This gives (2), and the fact that $H_2(N \setminus \nu(A)) \cong \Z \oplus G$, where $G$ is the image of the map $H_3(N,N \setminus \nu(A)) \to H_2(N \setminus \nu(A))$.

A priori, $G$ is some quotient of $\Z^{c+1}$. However, recall that the submanifold representing the extra generator described above (the Poincar\'{e}-Lefschetz dual of an arc running from one boundary component of $A$ to the other) has boundary a torus which can be taken to be isotopic to the boundary of a tubular neighborhood of either $\mu_\epsilon$. This is then nullhomologous in $N\setminus\nu(A)$ since it bounds a copy of $M_J\setminus\nu(\mu_J)$. Therefore, $G$ in fact just has one generator for each self-intersection of $A$, with possibly some relations.
\end{proof}

Notice that $N \setminus \nu(A)$ has two boundary components, with $\partial^+(N \setminus \nu(A)) = M_{P_0(J)\#-P_1(J)}$ unchanged from $\partial^+ N$. We wish to ``plug the hole'' by gluing in some 4-manifold with boundary equal to $\partial^-(N \setminus \nu(A))$. To that end, we will find a surgery description of this 3-manifold.

To start, take an arc in $N$ running from $\mu_0$ to $\mu_1$ along $A$ (recall that the $\mu_\epsilon$ are the copies of the meridian of $J$ living in $M_J\times\{1\}\subset(M_J \times I)_\epsilon$), missing all self-intersections of $A$. If we cut out a neighborhood of this arc from $N$, the new boundary is $M_J \# M_{-J}$, and the remainder of $A$ is an immersed disk bounded by the grey curve in Figure \ref{fig:SimpleBoundary}. When we remove the remainder of $A$, if there were no self-intersections, this would amount to doing 0-surgery on this grey curve. We can account for the self-intersections as in \cite{freedman4ManifoldTopologyII1995}, to get the surgery description shown in Figure \ref{fig:Boundary}. Each self-intersection corresponds to one of the Whitehead links which is banded on to the diagram. The sign of the clasp should match the sign of the self-intersection, but won't affect our discussion so we draw the clasps with indeterminate crossings following the convention in \cite{freedman4ManifoldTopologyII1995}. A priori, the immersed disk bounded by the grey curve in Figure \ref{fig:SimpleBoundary} might induce a nontrivial framing. However, as $A$ is an immersed annulus in $B^4$, the framing difference between the two ends must be twice the algebraic self-intersection number of $A$, similar to how adding a local kink changes the framing of an immersed disk by $\pm 2$. We can thus add local kinks to undo this framing difference, which makes the immersed disk bounded by the grey curve in Figure \ref{fig:SimpleBoundary} 0-framed.

Notice that doing 0-surgery on the grey curve in Figure \ref{fig:SimpleBoundary} produces $M_{J\#-J}$ (as in Figure \ref{fig:Zboundary}). Notice that the 3-manifold doesn't change if we ``blow up'' the clasps as in Figure \ref{fig:Zconstruction} (for now ignore the different labels on the components and instead treat everything as 0-framed). Again, the signs of the clasps match the signs of the corresponding self-intersections of $A$. Now, paying attention to the labels in Figure \ref{fig:Zconstruction}, we see that $\partial^-(N \setminus \nu(A))$ bounds a manifold $Z$ (framings written inside angle brackets denote a surgery description of a 3-manifold, which in this case is $M_{J\#-J}$), with schematic picture shown in Figure \ref{fig:Zconstruction2}.

\begin{figure}
	\begin{center}
	\includegraphics[width=0.5\linewidth]{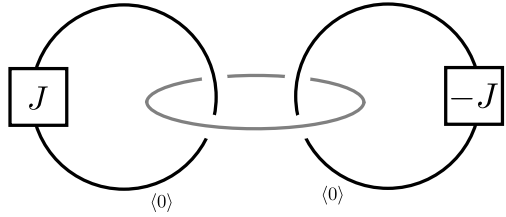}
	\end{center}
	\caption{$\partial^-(N\setminus\nu(\text{arc}))=M_J \# M_{-J}$ with the grey curve bounding an immersed disk obtained by cutting $A$ open along the arc}
	\label{fig:SimpleBoundary}
\end{figure}

\begin{figure}
	\begin{center}
	\includegraphics[width=0.5\linewidth]{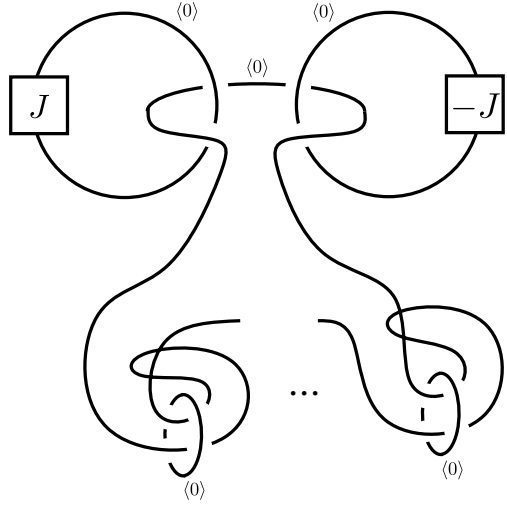}
	\end{center}
	\caption{$\partial^-(N \setminus \nu(A))$ accounting for self-intersections of $A$; the sign of each clasp matches the sign of the corresponding self-intersection}
	\label{fig:Boundary}
\end{figure}

\begin{figure}
	\begin{center}
	\includegraphics[width=0.5\linewidth]{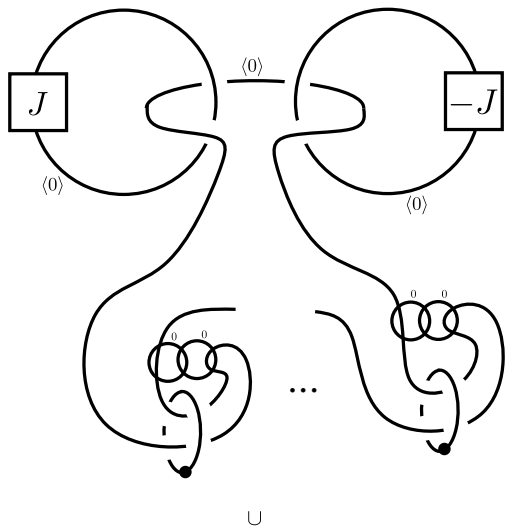}
	\end{center}
	\caption{Construction of $Z$}
	\label{fig:Zconstruction}
\end{figure}

It is straightforward to compute the homology of $Z$ and $\partial Z=\partial^-(N \setminus \nu(A))$, with the results as in the following two lemmas.

\begin{lemma}
\begin{enumerate}
\item $H_1(Z)=\Z^{c+1}$ where $c$ is the number of self-intersections of $A$, and the extra generator is the meridian of the slice disk for $J\#-J$.

\item $H_2(Z)=\Z^{2c}$, and is represented by surfaces which are the unions of: the cores of the 2-handles in $Z$, the homotopies through $M_{J\#-J}\times I$ which consist of undoing the clasps in the attaching circles of the 2-handles in the construction of $Z$, and parallel Seifert surfaces for $J$. See Figure \ref{fig:Zconstruction2} for a schematic picture.
\end{enumerate}
\end{lemma}

\begin{figure}
	\begin{center}
	\includegraphics[width=\linewidth]{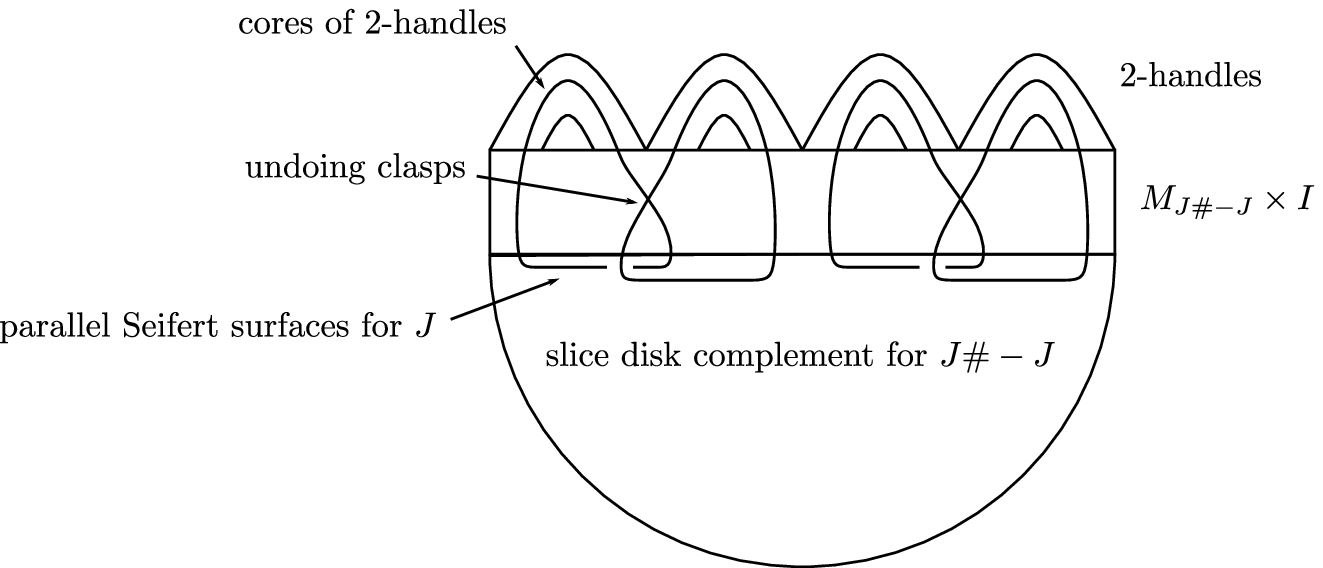}
	\end{center}
	\caption{Schematic picture of $Z$ with embedded surfaces generating $H_2(Z)$ (1-handles are not pictured)}
	\label{fig:Zconstruction2}
\end{figure}

\begin{lemma}
\begin{enumerate}
\item $H_1(\partial Z)=\Z^{c+1}$ (and $H_1(\partial Z) \to H_1(Z)$ is an isomorphism).
\item $H_2(\partial Z)=\Z^{c+1}$ (and $H_2(\partial Z) \to H_2(Z)$ is the zero map). Moreover, $H_2(\partial Z) \to H_2(N \setminus \nu(A))$ takes generators to corresponding generators.
\end{enumerate}
\end{lemma}

Now, let $N'=(N \setminus \nu(A)) \cup Z$.

\begin{lemma}
\begin{enumerate}
\item $H_1(N')\cong\Z$ and is generated by $\lambda$.
\item $H_2(Z) \to H_2(N')$ is an isomorphism.
\end{enumerate}
\label{N'Homology}
\end{lemma}

\setlength\mathsurround{0pt}

\begin{proof}
Consider the Mayer-Vietoris sequence:
	$$\begin{tikzcd}
		H_2(\partial Z) \arrow[r]
			& H_2(Z) \oplus H_2 (N \setminus \nu(A)) \arrow[r] \arrow[d, phantom, ""{coordinate, name=Z}]
			& H_2(N')  \arrow[dll,
											rounded corners,
											to path={ -- ([xshift=2ex]\tikztostart.east)
											|- (Z) [near end]\tikztonodes
											-| ([xshift=-2ex]\tikztotarget.west)
											-- (\tikztotarget)}] \\
		H_1(\partial Z) \arrow[r]
			& H_1(Z) \oplus H_1(N \setminus \nu(A)) \arrow[r]
			& H_1(N') \arrow[r]
			& 0
	\end{tikzcd}.$$

\setlength\mathsurround{0.8pt}

Since $H_1(\partial Z) \to H_1(Z)$ is an isomorphism, we see that $H_1(N \setminus \nu(A)) \to H_1(N')$ is an isomorphism. Composing with isomorphisms from before we get (1). We also see that $H_2(N') \to H_1(\partial Z)$ is the zero map.

Since $H_2(\partial Z) \to H_2(Z)$ is also the zero map, and $H_2(\partial Z) \to H_2(N \setminus \nu(A))$ takes generators to generators, we see that $H_2(Z) \to H_2(N')$ is an isomorphism.
\end{proof}

We now prove Theorem \ref{1solvable}.

\begin{proof}[Proof of Theorem \ref{1solvable}]
By the proof of the previous theorem, $H_2(N')$ is generated by the same surfaces that generate $H_2(Z)$. These surfaces consist of Seifert surfaces which are capped off by 0-framed 2-handles (along with the homotopies that preserve the framing), and so have trivial normal bundle. They intersect geometrically once in pairs, and the generators of their fundamental groups lie on a Seifert surface for $J$, and thus are in the commutator subgroup of $\pi_1(Z)$, which has a natural homomorphism into $\pi_1(N')$. Therefore $N'$ is a (1)-solution for $P_0(J)\#-P_1(J)$.
\end{proof}

\begin{remark}
\label{nSolvable}
If the map on fundamental groups induced by the inclusion from the Seifert surface for $J$ to the slice disk complement for $J\#-J$ has image in the $n$th term of the derived series, then $N'$ would be an $(n)$-solution for $P_0(J)\#-P_1(J)$. We did not have to use the usual slice disk for $J\#-J$, and one might be able to get a better result by choosing a different slice disk for $J\#-J$. In Section \ref{SlicenessConditions}, we will give even stronger conditions that would guarantee that $P_0(J)\#-P_1(J)$ is topologically slice.
\end{remark}

Given a slice knot, one might wonder if it is ribbon:

\begin{definition}
	A (smoothly) slice knot $K$ is {\it ribbon} if it has a (smooth) slice disk $\Delta$ such that the restriction of radial Morse function on $B^4$ to $\Delta$ has only index 0 and 1 critical points.
\end{definition}

This gives rise to the following famous conjecture:

\begin{conj}[Slice-Ribbon Conjecture]
	Every slice knot is ribbon.
\end{conj}

One can weaken ribbonness to the following algebraic topological condition, which makes sense in both the smooth and topological settings (unlike ribbonness, which relies on a Morse function).

\begin{definition}
	A (smoothly/topologically) slice knot is \it{(smoothly/topologicaly) homotopy ribbon} if it bounds a (smooth/topological) slice disk $\Delta$ such that the inclusion-induced map $\pi_1(M_K)\to\pi_1(B^4\setminus\nu(\Delta))$ is surjective.
\end{definition}

Every ribbon knot is homotopy ribbon. This gives rise to the topological analogue of the Slice-Ribbon Conjecture:

\begin{conj}
	Every topologically slice knot is topologically homotopy ribbon.
\end{conj}

Smooth homotopy ribbonness can be extended to $(n)$-solutions:

\begin{definition}
	An $(h)$-solvable knot (for $h\in\frac{1}{2}\N$) is \it{homotopy ribbon $(h)$-solvable} if it bounds an $(h)$-solution $W$ such that the inclusion map $\pi_1(M_K) \to \pi_1(W)$ is surjective. The manifold $W$ is then called a \it{homotopy ribbon $(h)$-solution}.
\end{definition}

Note that the homotopy ribbon $(h)$-solution must be a smooth manifold.

The $(1)$-solution we have constructed for $P_0(J)\#-P_1(J)$ is in fact a homotopy ribbon $(1)$-solution when $K_0\#-K_1$ is smoothly homotopy ribbon:

\begin{prop}
\label{HomotopyRibbon1Solvable}
In the notation of Definition \ref{HomotopicallyRelated} and Theorem \ref{1solvable}, if $K_0\#-K_1$ is smoothly homotopy ribbon via the disk $C\setminus(\text{an arc running from $K_0$ to $K_1$})$ (where $C$ is the concordance between $K_0$ and $K_1$), then the $(1)$-solution $N'$ for $P_0(J)\#-P_1(J)$ constructed in Theorem \ref{1solvable} is in fact a homotopy riboon $(1)$-solution.
\end{prop}

Before proving this proposition, we will need the following lemma:

\begin{lemma}
\label{RibbonManifoldArcs}
Given a 4-manifold $W$ with connected boundary $M$ such that the inclusion map $\pi_1(M) \to \pi_1(W)$ is surjective, any arc $\alpha$ in $W$ with both ends in $M$ may be homotoped to an arc in $M$ fixing the endpoints of the arc.
\end{lemma}

\begin{proof}
Choose the basepoint of both $\pi_1(M)$ and $\pi_1(W)$ to be a point $p \in M$. Choose whiskers $v,w$ from the basepoint to each of the ends of $\alpha$ inside $M$. This gives a based loop $\gamma = v \alpha w^{-1}$. The loop $\gamma$ is then homotopic (rel basepoint) to a loop $\gamma'$ lying entirely in $\partial W = M$, ie there is a continuous map from an annulus to $W$ which on the boundary is $\gamma \cup \gamma'$. We can then see that the arc $v^{-1} \gamma' w$, is homotopic rel boundary to $\alpha$, by homotoping along the map from an annulus (See Figure \ref{fig:RibbonMfldArc}).
\end{proof}

\begin{figure}
	\begin{center}
	\includegraphics{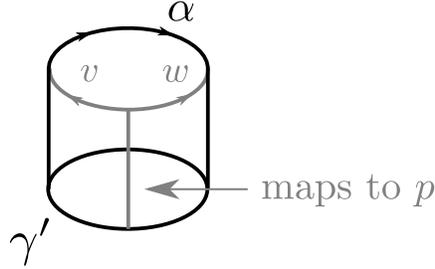}
	\end{center}
	\caption{The arc $\alpha$ is homotopic fixing endpoints to the arc travelling along one whisker, then traversing $\gamma'$, and finally travelling back along the other whisker.}
	\label{fig:RibbonMfldArc}
\end{figure}

\begin{proof}[Proof of Proposition \ref{HomotopyRibbon1Solvable}]
Recall that $N'=(N \setminus \nu(A)) \cup Z$, where a schematic for $N$ is depicted in Figure \ref{fig:NConstruction}, a partial Kirby diagram for $Z$ is depicted in Figure \ref{fig:Zconstruction}, and a schematic for $Z$ is depicted in Figure \ref{fig:Zconstruction2}. Take any based loop $\gamma \in \pi_1(N')$. We will homotope this curve to lie entirely in $\pi_1(M_{P_0(J)\#-P_1(J)})$, and abusing notation, will also denote each step along the way by $\gamma$. Without loss of generality, the intersection of this curve with $Z$ is a collection of arcs.

Since $Z$ consists of the complement of a ribbon disk for $J\#-J$, with some 1- and 2-handles attached, one can see that $\pi_1(\partial Z) \to \pi_1(Z)$ is surjective. This can be seen since inclusion-induced map from the boundary of a ribbon disk complement is already surjective, and a standard argument shows that attaching 1- and 2-handles will preserve surjectivity. Therefore, using Lemma \ref{RibbonManifoldArcs}, each of the arcs may be homotoped fixing their endpoints to lie in $\partial Z$, and then pushed off to lie entirely in $N \setminus \nu(A)$. The intersection of $\gamma$ with the $(M_J \times I)_\epsilon$ can pushed straight outward to $(M_J\times\{0\})_\epsilon$. Since the attaching regions between $W$ and the $(M_J \times I)_\epsilon$ are neighborhoods of curves, we can perturb $\gamma$ slightly so that it does not intersect them. At this point $\gamma$ lies entirely in $(W \setminus \nu(A)) \cup \partial N'$.

Now, the intersections of $\gamma$ with $W \setminus \nu(A)$ without loss of generality again consist of a collection of arcs. Since $K_0\#-K_1$ is homotopy ribbon, these arcs could be homotoped to $\partial W$, again using Lemma \ref{RibbonManifoldArcs} (since $(S^3 \times I)\setminus\nu(C)$ is homeomorphic to the exterior of a homotopy ribbon disk for $K_0\#-K_1$), but might hit $A$ along the way. However, this is no problem, as a meridian of $A$ can be homotoped to curve in $\partial Z$ which isotopic through $Z$ to the attaching curve of any of the 2-handles (it is a meridian of unknotted curve with framing $\langle 0 \rangle$), and so in particular bounds a disk. Therefore, we can replace neighborhoods of potential intersections with $A$ with these disks, thus achieving the desired homotopy.

So, now $\gamma$ lies entirely in $\partial(N')=M_{P_0(J)\#-P_1(J)}$, and we are done.
\end{proof}

\begin{remark}
	Note that this proof holds as long as we use a smooth homotopy ribbon disk for $J\#-J$ in the proof of Theorem \ref{1solvable}. In particular, as long as we use a (smooth) homotopy ribbon disk for $J\#-J$, Remark \ref{nSolvable} will still apply. Moreover, Proposition \ref{SliceGuarantee} below will still apply even if we only use a topological homotopy ribbon disk for $J \# -J$, though it is worth noting that this involves an application of the Sphere Embedding Theorem up to $s$-cobordism with a $\pi_1$-null condition and topological input. See Theorem 6.1 of \cite{freedmanTopology4manifolds1990}, as well as the discussion at the beginning of Chapter 5, and Sections 5.2 and 5.3.
\end{remark}

\section{Some conditions that would guarantee topological sliceness}
\label{SlicenessConditions}

\begin{prop}
\label{SliceGuarantee}
If there is a smooth or topological slice disk complement for $J\#-J$ in which any number of parallel longitudes of $J$ bound disjoint framed $\pi_1$-null immersed disks with algebraically trivial self-intersections, then $P_0(J)\#-P_1(J)$ is topologically slice.
\end{prop}

Before proving this proposition we will review some definitions and techniques from the theory of 4-manifolds. These make sense in both the smooth and topological settings.

\begin{definition}
An $h$-cobordism rel boundary between two manifolds $M^n$ and $N^n$ with homeomorphic (possibly empty) boundary $P$ is a manifold $W^{n+1}$ with boundary $M \cup (P \times I) \cup N$ where $M$ and $N$ are glued to $P \times I$ along their boundaries, such that $W$ deformation retracts to $M$ and to $N$. Given homeomorphic properly immersed submanifolds $A \looparrowright M$, $B \looparrowright N$, an $h$-cobordism rel boundary of the pairs $(M,A)$, $(N,B)$ is an $h$-cobordism rel boundary $W$ of $M$ and $N$ together with a proper immersion $A \times I \looparrowright W$ which is a product on $P \times I$ and such that the restriction to $A\times\{0\}$ gives $A \looparrowright M$ and the restriction to $A \times \{1\}$ gives $B \looparrowright N$. See Figure \ref{fig:PairHCobordism} for a schematic picture.
\end{definition}

\begin{figure}
	\centering
	\includegraphics{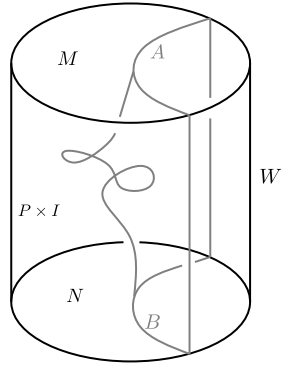}
	\caption{A schematic picture of an $h$-cobordism rel boundary of pairs. Note that $A$ and $B$ may also be immersed, but must be homeomorphic.}
	\label{fig:PairHCobordism}
\end{figure}

\begin{definition}
An $s$-cobordism (rel boundary of pairs) is an $h$-cobordism (rel boundary of pairs) such that the Whitehead torsion $\tau(W,M)$ vanishes. For a definition and discussion of the Whitehead torsion in the topological setting, see Section IV of \cite{cohenCourseSimplehomotopyTheory1973}. For a definition and discussion in the smooth setting, see \cite{milnorWhiteheadTorsion1966}.
\end{definition}

Now we will describe a 4-dimensional analog of 0-surgery. Given a sphere embedded in a 4-manifold $S^2 \hookrightarrow W^4$ with trivial normal bundle (that is, the given embedding extends to an embedding of $S^2 \times D^2$), we can cut out a neighborhood of the sphere. The new boundary is naturally homeomorphic to $S^2 \times S^1$. We then replace the piece we cut out with $D^3 \times S^1$, gluing along the identity homeomorphism $S^2 \times S^1 \to S^2 \times S^1$. If we are given two such spheres which intersect once transversely, each of which represents a generator of a $\Z$ summand of $H_2(W)$, then by Seifert-van Kampen and Mayer-Vietoris, doing surgery on one of the two spheres exactly kills the $\Z^2$ summand in $H_2(W)$ and leaves $\pi_1(W)$ and the rest of the homology groups unchanged.

\begin{proof}[Proof of Proposition \ref{SliceGuarantee}]
In the construction of $Z$ depicted in Figures \ref{fig:Zconstruction} and \ref{fig:Zconstruction2}, use the slice disk complement for $J\#-J$ given by the the conditions of Proposition \ref{SliceGuarantee}. (For the purposes of Theorem \ref{1solvable}, we could have used any slice disk complement for $J\#-J$.) And moreover, the meridian of $\eta$, together with the longitudes of $J$ which are freely homotopic to the attaching circles of the 2-handles, all bound disjoint framed $\pi_1$-null immersed disks in the slice disk complement (these are all parallel longitudes of $J$). Since the meridian of $\eta$ is then freely nullhomotopic in $Z$, we see that $\pi_1(N')$ is normally generated by $\lambda$. By the topological version of Theorem 6.1 in \cite{freedmanTopology4manifolds1990}, there is an $s$-cobordism rel boundary of the pair $(N', \text{immersed spheres})$ to a pair $(N'', \text{embedded spheres})$, where the embedded spheres still come in pairs which intersect once transversely. (See the discussion at the beginning of Chapter 5 in \cite{freedmanTopology4manifolds1990} for how to extend Theorem 6.1 to the topological setting.) Moreover, since the original disks were framed, these embedded spheres have trivial normal bundles. We can then do surgery on one sphere from each pair as described above to get a manifold that satisfies the conditions in Proposition \ref{TopSliceCondition}, and so $P_0(J)\#-P_1(J)$ is topologically slice.
\end{proof}

\begin{figure}
	\begin{center}
	\includegraphics[width=0.5\linewidth]{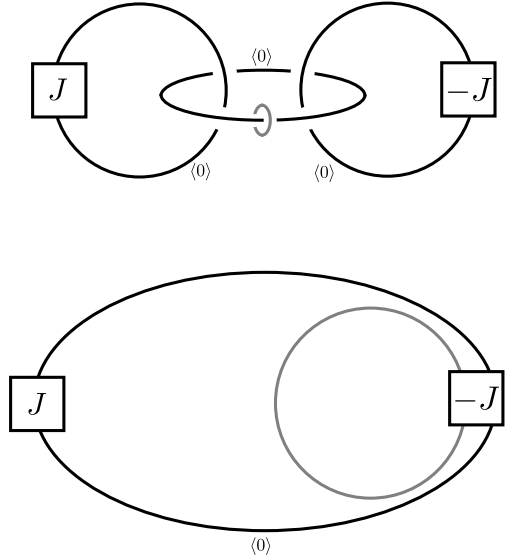}
	\end{center}
	\caption{Boundary of a slice disk complement for $J\#-J$. Any number of parallel copies of the grey curve must bound disjoint immersed $\pi_1$-null disks in order to prove $P_0(J)\#-P_1(J)$ is slice.}
	\label{fig:Zboundary}
\end{figure}

\begin{remark}
In the usual slice disk complement for $J\#-J$, the parallel longitude $J$ is not nullhomotopic, as the slice disk complement is homeomorphic to $(S^3 \setminus J) \times I$ and the longitude of a knot is not nullhomotopic in its complement. So, in order for the conditions of Corollary \ref{SliceGuarantee} to be satisfied, we would need to find a different (potentially topological) slice disk. A promising approach might be the work of Friedl-Teichner and Conway in \cite{friedlNewTopologicallySlice2005} and \cite{conwayHomotopyRibbonDiscs2022}. The first gives an algebraic condition on a surjection $\pi_1(M_K) \twoheadrightarrow G$ which guarantees that $K$ has a topological homotopy ribbon disk with fundamental group $G$ (this is a generalization of the theorem of Freedman that any knot with Alexander polynomial 1 is topologically $\Z$-slice, as their conditions reduces to the Alexander polynomial 1 condition for the abelianization map $\pi_1(M_K)\to\Z$). Conway expands on this, giving an algebraic classification of topological homotopy ribbon disks corresponding to a given surjection $\pi_1(M_K) \twoheadrightarrow G$. So, if we could find a surjection $\pi_1(M_{J\#-J}) \twoheadrightarrow G$ for some group $G$ which satisfies the conditions in the paper and which kills the longitude of $J$, this would be a large step in producing a slice disk which satisfies the conditions of Corollary  \ref{SliceGuarantee}. However, it is usually quite difficult to tell whether a particular surjection $\pi_1(M_{J\#-J}) \twoheadrightarrow G$ satisfies the relevant algebraic conditions.
\end{remark}

\section{Casson-Gordon invariants \label{CGinvariants}}

We will now review the setup for Casson-Gordon invariants, first defined in \cite{cassonCobordismClassicalKnots1986}. For any knot $K$, let $L_{K,n}$ be the $n$-fold cyclic branched cover of $S^3$ over $K$. Let $\chi:H_1(L_{K,n})\to\Z_m$ be a character. We will also use $\chi$ to denote the postcomposition of this map with the map $\Z_m\to\C^*$ sending $1 \mapsto e^{\frac{2 \pi i}{m}}$. To each such character, Casson and Gordon associate a {\it Witt class} $\tau(K,\chi) \in L_0(\C(t))\otimes\Q$.

	The {\it Witt group} $L_0(k)$ for a field $k$ is the set of equivalence classes of nondegenerate symmetric bilinear forms where two forms are equivalent if one can direct sum a metabolic form to each and get isometric forms. A metabolic form is a form which has a half-dimensional subspace on which it vanishes. Such a subspace is called a metabolizer. The set of nondegenerate symmetric bilinear forms under this equivalence relation forms a group under the direct sum. For more information on Witt groups, see e.g. Chapter 5 of \cite{livingstonIntroductionKnotConcordance}.

The following is a theorem of Casson-Gordon:

\begin{theorem}[Theorem 2 in \cite{cassonCobordismClassicalKnots1986}]
\label{CGSliceObstr}
Fix $n$ a prime power. If $K$ is smoothly or topologically slice, then the linking form on $H_1(L_{K,n})$ has a metabolizer $R$, and, for any character $\chi$ with prime power order $m$ which vanishes on $R$, we have $\tau(K,\chi)=0$.
\end{theorem}

The sliceness condition on $K$ can be weakened to (1.5)-solvability: 

\begin{theorem}[see Theorem B.1 of \cite{cochranFilteringSmoothConcordance2013} and Theorem 9.11 of \cite{cochranKnotConcordanceWhitney2003}]
\label{CG1.5SolvObstr}
The conclusion in Theorem \ref{CGSliceObstr} holds even if $K$ is only (1.5)-solvable.
\end{theorem}

So, Casson-Gordon invariants give us an obstruction to (1.5)-solvability, which we will show fails for $P_0(J)\#-P_1(J)$.

\begin{prop}
\label{CGObstrFails}
The conclusion of Theorems \ref{CGSliceObstr} and \ref{CG1.5SolvObstr} hold for $P_0(J)\#-P_1(J)$. In other words, Theorem \ref{CG1.5SolvObstr} fails to obstruct the $(1.5)$-solvability of the knot $P_0(J)\#-P_1(J)$.
\end{prop}

Before proving this proposition, we will review some results of Litherland about the Casson-Gordon invariants of satellite knots. For reference, see \cite{litherlandCobordismSatelliteKnots1984} and \cite{litherlandSymmetriesTwistSpunKnots1985}. For a satellite operator $P$ given by $K\cup\eta$, let $w=\lk(K,\eta)$ be the winding number of $P$. Let $h=\gcd(n,w)$ and let $k=\frac{n}{h}$. Write $\Ch_n(K)=\Hom(H_1(L_{K,n}),\C^*)$ for the group of characters. There are canonical isomorphisms
	$$H_1(L_{P(J),n}) \cong H_1(L_{K,n}) \oplus h H_1(L_{J,k})$$
and
	$$\Ch_n(P(J)) \cong \Ch_n(K) \oplus h\Ch_k(J)$$
where $hA$ is the direct sum of $h$ copies of the group $A$, and the linking form on $H_1(L_{P(J),n})$ is given by the direct sum of the linking forms on the summands. (When $k=1$, the 1-fold branched cover $L_{J,1}$ is just $S^3$, and so $\Ch_1(J)$ is the trivial group.)

Then, given a character $\chi\in\Ch_n(P_\epsilon(J))$ which is identified under this isomorphism with $(\chi_K,\chi_1,\cdots,\chi_h)$, we have:
	$$\tau(P(J),\chi)[t]=\tau(K,\chi_K)[t]+\sum_{i=1}^{h}\tau(J,\chi_i)[\chi_K(\tilde{\eta}^i)t^{\frac{w}{h}}]$$
where the $\tilde{\eta}^i$ are the $h$ lifts of of $\eta$ to $L_{K,n}$ (or more precisely, the $h$ lifts of $k\eta$).

Moreover, for any knot $G$, the map $\cdot \# G$ is the satellite operator given by $G\cup\mu_G$ where $\mu_G$ is the meridian of $G$. In this case, $w=h=1$ and $k=n$, so we get
	$$\tau(F\#G,\chi)=\tau(F,\chi_F)+\tau(G,\chi_G)$$
(note that $\tilde\mu_G$ bounds a disk and in particular is nullhomologous in any cyclic branched cover over $G$, and so $\chi_G(\tilde\mu_G)=1\in\C^*$).

\begin{proof}[Proof of Proposition \ref{CGObstrFails}]
From the discussion above, we get the following canonical isomorphisms:
\begin{align*}
	H_1(L_{P_0(J) \# -P_1(J), n})
		& \cong H_1(L_{P_0(J),n}) \oplus H_1(L_{-P_1(J),n})\\
		& \cong H_1(L_{K_0,n}) \oplus h H_1(L_{J,k}) \oplus H_1(L_{-K_1,n}) \oplus h H_1(L_{-J,k})\\
		& \cong H_1(L_{K_0 \# -K_1, n}) \oplus h H_1(L_{J \# -J, k})
\end{align*}
Similarly, we have a canonical isomorphisms
\begin{align*}
	\Ch_n(P_0(J)\#-P_1(J))
		& \cong \Ch_n(K_0) \oplus h \Ch_k(J) \oplus \Ch_n(-K_1) \oplus h \Ch_k(-J)\\
		& \cong \Ch_n(K_0\#-K_1) \oplus h \Ch_k(J\#-J)
\end{align*}
	
By Theorem \ref{CGSliceObstr}, the linking forms on branched cyclic covers of $K_0\#-K_1$ and $J\#-J$ have metabolizers since these knots are topologically slice. Denote the metabolizer of $H_1(L_{K_0\#-K_1,n})$ by $R$ and the metabolizer of $H_1(L_{J\#-J,l})$ by $S$. The image of the submodule $R \oplus hS$ under the canonical isomorphisms is thus a metabolizer for the linking form on $H_1(L_{P_0(J) \# -P_1(J), n})$. A character $\chi\in\Ch_n(P_0(J)\#-P_1(J))$ vanishes on $R \oplus hS$ if and only if $\varphi$ vanishes on $R$ and each $\psi^i$ vanishes on $S$, where $\varphi \oplus \bigoplus_{i=1}^{h} \psi^i \in \Ch_1(L_{K_0\#-K_1,n}) \oplus h \Ch_1(L_{J\#-J,k})$ corresponds to $\chi$.

Now, let $\chi$ be a character which vanishes on $R \oplus hS$. Using the canonical isomorphisms above, decompose $\chi$ as
	$$\chi_{P_0(J)}\oplus\chi_{-P_1(J)} = \chi_{K_0} \oplus \bigoplus_{i=1}^{h} \chi_{J}^i \oplus \chi_{-K_1} \oplus \bigoplus_{i=1}^{h} \chi_{-J}^i$$
where the subscript denotes the knot whose character group to which the character belongs.

Now, we have
\begin{align*}
	\tau(P_0(J)\#-P_1(J),\chi)[t]
		& = \tau(P_0(J),\chi_{P_0(J)})[t]+\tau(-P_1(J),\chi_{-P_1(J)})[t]\\
		& = \tau(K_0,\chi_{K_0})[t]+\sum_{i=1}^{h}\tau(J,\chi_{J}^{i})[\chi_{K_0}(\tilde{\eta}_0^i)t^{\frac{w}{h}}]\\
		& \quad +\tau(-K_1,\chi_{-K_1})[t]+\sum_{i=1}^{h}\tau(-J,\chi_{-J}^{i})[\chi_{-K_1}(-\tilde{\eta}_1^i)t^{\frac{w}{h}}].
\end{align*}
By $\tilde\eta_\epsilon^i$ we are denoting the lifts of $\eta_\epsilon$ for $\epsilon\in\{0,1\}$.

Now we wish to show that $\chi_{K_0}(\tilde\eta_0^i)=\chi_{-K_1}(-\tilde\eta_1^i)$, after possibly reindexing the lifts of the $\eta_\epsilon$s. First, consider the $n$-fold cyclic branched cover of $B^4$ over the natural slice disk for $K_0\#-K_1$. The boundary of this branched cover is the $n$-fold cyclic branched cover of $K_0\#-K_1$. Now, by the homotopy lifting property, the $h$ lifts of $\eta_0$ are each joined to one of the $h$ lifts of $-\eta_1$ through an immersed annulus in the branched cover of $B^4$ which is a lift of $A$ (these lifts are all $k$-fold covers). Therefore, after possibly reindexing, we have that $\tilde\eta_0^i\oplus-(-\tilde\eta_1^i)$ is nullhomologous in the branched cover of $B^4$, so is in the kernel of the inclusion induced map, and therefore lies in the metabolizer $R$ of the linking form on $H_1(L_{K_0\#-K_1,n})$. (We write two minus signs since the $-\tilde\eta_1^i$ are lifts of $-\eta_1$, and $-(-\tilde\eta_1^i)$ are the opposites of the homology classes of these lifts.)

Now, since $\varphi$ vanishes on the metabolizer $R$ of $H_1(L_{K_0\#-K_1})$, we have that
\begin{align*}
	0
		& = \varphi(\tilde\eta_0^i\oplus-(-\tilde\eta_1^i))\\
		& = \chi_{K_0}(\tilde\eta_0^i)+\chi_{-K_1}(-(-\tilde\eta_1^i))\\
		& =\chi_{K_0}(\tilde\eta_0^i)-\chi_{-K_1}(-\tilde\eta_1^i),
\end{align*}
and so indeed $\chi_{K_0}(\tilde\eta_0^i)=\chi_{-K_1}(-\tilde\eta_1^i)$.

Therefore, we have that
\begin{align*}
	\tau(P_0(J)\#-P_1(J),\chi)
		& = \tau(K_0\#-K_1,\varphi)[t]+\sum_{i=1}^{h}\tau(J\#-J,\psi)[\chi_{K_0}(\tilde\eta_0^i)t^{\frac{w}{h}}]\\
		& = 0 + \sum_{i=1}^{h}0 = 0
\end{align*}
since $K_0\#-K_1$ and $J\#-J$ are both topologically slice and the corresponding characters vanish on the metabolizer by tracing through the isomorphism on $H_1$ (we could have equivalently written $\chi_{-K_1}(-\tilde\eta_1^i)$ instead of $\chi_{K_0}(\tilde\eta_1^i)$ in the final equation).
\end{proof}

\section{Metabelian \texorpdfstring{$\rho$}{rho} invariants}
\label{MetaRhoInvs}

Recall that the rational Alexander module $\cA_0(K)=H_1(M_K,\Q[t^{\pm 1}])$ comes equipped with the Blanchfield form
	$$\Bl_0:\cA_0(K) \times \cA_0(K) \to \Q(t)/\Q[t^{\pm 1}]$$
defined in the following way: since $\cA_0(K)$ is $\Q[t^{\pm 1}]$-torsion, given homology classes $[x],[y]\in\cA_0(K)$, there is a Laurent polynomial $p\in\Q[t^{\pm 1}]$ such that $py$ is nullhomologous. That is, there is some 2-chain $w \in C _2(M_K,\Q[t^{\pm 1}])$ such that $py = \partial w$. Then,
	$$\Bl_0([x],[y])=\frac{1}{p}\sum_{n\in\Z}\langle t^n x,w \rangle t^{-n}.$$
This can be thought of as a linking form which is equivariant with respect to the deck action.
	
For a 3-manifold $M$ (for our purposes, this will always be 0-surgery on a knot) and a representation $\varphi:\pi_1(M)\to\Gamma$, Cheeger and Gromov defined {\it von Neumann $\rho$-invariants} $\rho(M,\varphi)$ in \cite{cheegerBoundsNeumannDimension1985}. Typically, we will assume that $\Gamma$ is {\it poly-torsion-free-abelian (PTFA)}, which means that it admits a normal series
$$\{1\}\lhd\Gamma_1\lhd\cdots\lhd\Gamma_{n-1}\lhd\Gamma_n=\Gamma$$
such that each quotient $\Gamma_{i+1}/\Gamma_i$ is torsion-free abelian. If $M$ bounds a 4-manifold $W$, and $\psi$ extends to $\pi_1(W)$, then $\rho(M,\varphi)$ is given by the difference $\sigma_\Gamma^{(2)}(W,\psi)-\sigma(W)$, where $\sigma(W)$ is the ordinary signature of $W$ and $\sigma_\Gamma^{(2)}(W,\psi)$ is a certain kind of twisted signature. In other words, the $\rho$-invariants can be thought of as ``signature defects.'' We will not define them rigorously here, and instead will use a few key properties:

\begin{itemize}
	\item Given an injection $\Gamma\hookrightarrow\Gamma'$, we get a map $\varphi':\pi_1(M)\to\Gamma\hookrightarrow\Gamma'$. In this situation, $\rho(M,\varphi)=\rho(M,\varphi')$.
	\item If $\varphi$ is the zero map, then $\rho(M,\varphi)=0$.
	\item If $K$ is slice, $\Gamma$ is a PTFA group, and $\varphi:\pi_1(M_K)\to\Gamma$ extends to the fundmental group of the exterior of a slice disk, then $\rho(M_K,\varphi)=0$.
\end{itemize}

We will also abuse notation slightly, writing $\rho(K,\varphi):=\rho(M_K,\varphi)$.
	
Given a submodule $P\subset\cA_0(K)$, we get a normal subgroup of $\pi_1(M_K)$:
	$$\tilde{P}=\ker(\pi_1(M_K)^{(1)} \twoheadrightarrow \pi_1(M_K)^{(1)}/\pi_1(M_K)^{(2)}=\cA(K) \to \cA_0(K) \twoheadrightarrow \cA_0(K)/P)$$
where $\cA(K)=H_1(M_K;\Z[t^{\pm 1}])$ is the ordinary Alexander module. A priori, $\tilde{P}$ is only normal in $\pi_1(M_K)^{(1)}$. But, since $\pi_1(M_K)=\pi_1(M_K)^{(1)}\rtimes\Z$, we can write any element of $\pi_1(M_K)$ as $t^k a$ where $t$ is the generator of $\Z$ and $a\in\pi_1(M_K)^{(1)}$. Then, given an element $p\in\tilde{P}$, we see that $(t^k a)p(t^k a)^{-1}=t^k (apa^{-1}) t^{-k}$. Thus, since $\tilde{P}$ is normal in $\pi_1(M_K)^{(1)}$, we see that $apa^{-1}\in\tilde{P}$, and therefore its image in $\cA_0(K)$ is in $P$. But then, since $P$ is a submodule of $\cA_0(K)$, the element $t^k (apa^{-1}) t^{-k} = t^k \cdot apa^{-1}$ must also be in $P$, and therefore $(t^k a)p(t^k a)^{-1}\in\tilde{P}$. Thus, $\tilde{P}$ is normal in the larger group $\pi_1(M_K)$.

The normal subgroup $\tilde{P}$ in turn gives us a quotient map $\varphi_{\tilde{P}}:\pi_1(M_K)\twoheadrightarrow\pi_1(M_K)/\tilde{P}$. Since $\pi_1(M_K)^{(2)}\subseteq\tilde{P}$, this quotient is metabelian. As in \cite{cochranDerivativesKnotsSecondorder2010}, the groups $\pi_1(M_K)/\tilde{P}$ are all PTFA, and so we will not need to worry about this condition in the following discussion.

Recall from Section \ref{CGinvariants} that given a form on some space (for example the Blanchfield form $\Bl_0$ on the rational Alexander module $\cA_0(K)$ of a knot $K$), a {\it metabolizer} is a subspace which is its own orthogonal complement. Metabolizers of the Blanchfield form are intimately related with sliceness of knots: if $K$ is slice with a slice disk $\Delta$, then
	$$P_\Delta=\ker(\cA_0(K)=H_1(M_K;\Q[t^{\pm 1}]) \to H_1(B^4\setminus\Delta;\Q[t^{\pm 1}])$$
is a metabolizer for the Blanchfield form on $\cA_0(K)$.

The following is a theorem of Cochran-Harvey-Leidy:

\begin{theorem}[Theorem 4.2 in \cite{cochranDerivativesKnotsSecondorder2010}]
If $K$ is slice, then for some metabolizer $P_\Delta\subset\cA_0(K)$ corresponding to a slice disk $\Delta$, we have $\rho(K,\varphi_{\tilde{P}_\Delta})=0$. In particular, the set
	$$\{\rho(K,\varphi_{\tilde{P}}) | P \text{ is a metabolizer for } \cA_0(K)\}$$
contains 0.
\end{theorem}

Cochran-Harvey-Leidy showed that, just as with Casson-Gordon invariants, the condition of sliceness in the previous theorem can be weakened to (1.5)-solvability:

\begin{theorem}[Theorem 10.1 in\cite{cochranDerivativesKnotsSecondorder2010}]
\label{RhoInvObstr}
If $K$ is (1.5)-solvable, then the set
	$$\{\rho(K,\varphi_{\tilde{P}}) | P \text{ is a metabolizer for } \cA_0(K)\}$$
contains 0.
\end{theorem}

Similarly to Casson-Gordon invariants, we can write the metabelian $\rho$ invariants of a satellite knot $P(J)$ in terms of the $\rho$ invariants of $K=P(U)$ and $J$. For this we will need to choose a basepoint common to $M_{P(J)}$, $M_K \setminus \nu(\eta)$, and $S^3 \setminus \nu(J)$, as well as meridians and longitude for the boundaries of $M_K \setminus \nu(\eta)$ and $S^3 \setminus \nu(J)$. Our choices are shown in Figure \ref{fig:BaseptChoice} for the particular case of the Mazur pattern. We will implicitly carry this choice of basepoint with us throughout the following discussion.

\begin{figure}
	\centering
	\includegraphics[width=0.5\linewidth]{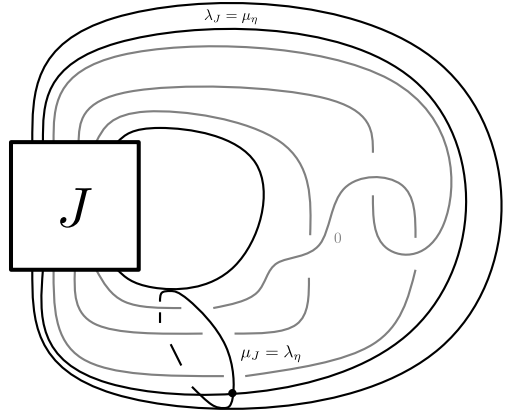}
	\caption{Decomposition of $M_{P_J}$ as $M_K \setminus \nu(\eta)$ and $S^3 \setminus J$ in the particular case of the Mazur pattern. The torus, with $K$ sitting inside, should be tied into $J$, and so it bounds the two submanifolds on either side. The basepoint is the intersection of the two labeled curves.}
	\label{fig:BaseptChoice}
\end{figure}

Let $\Gamma$ be a metabelian group. Given a homomorphism $\psi:\pi_1(M_{P(J)})\to\Gamma$, we get homomorphisms $\psi_K:\pi_1(M_K\setminus\nu(\eta))\to\Gamma$ and $\psi_J:\pi_1(S^3 \setminus \nu(J)) \to \Gamma$ since $M_K\setminus\nu(\eta)$ and $S^3 \setminus \nu(J)$ are both submanifolds of $M_{P(J)}$. When gluing together $M_K\setminus\nu(\eta)$ and $S^3\setminus\nu(J)$ to get $M_{P(J)}$ a meridian $\mu_\eta$ of $\eta$ and a longitude $\lambda_J$ of $J$ are identified. And, since $\lambda_J$ lies in $\pi_1(S^3 \setminus \nu(J))^{(2)}$, which includes into $\pi_1(M_{P(J)})^{(2)}$, we must have that $\psi$ sends the image of $\lambda_J$ (which is also the image of $\mu_\eta$) to $0\in\Gamma$. Therefore $\psi_K$ and $\psi_J$ extend uniquely to homomorphisms $\pi_1(M_K)\to\Gamma$ and $\pi_1(M_J)\to\Gamma$, which we will also denote $\psi_K$ and $\psi_J$, respectively. Similarly, we have $\psi_K(\eta)=\psi_J(\mu_J)$.

Conversely, given homomorphisms $\psi_K:\pi_1(M_K)\to\Gamma$ and $\psi_J:\pi_1(M_J)\to\Gamma$ such that $\psi_K(\eta)=\varphi_J(\mu_J)$, we can restrict them to $M_K\setminus\nu(\eta)$ and $S^3 \setminus \nu(J)$, and they will be compatible on the torus boundary when we glue these two manifolds together to get $M_{P(J)}$ (notice that the meridian of $\eta$ is nullhomotopic in $M_K$ and the longitude of $J$ is in $\pi_1(M_J)^{(2)}$ and so both must be sent to $1\in\Gamma$). Thus, $\psi_K$ and $\psi_J$ give a homomorphism $\psi:\pi_1(M_{P(J)})\to\Gamma$ whose restrictions to $M_K\setminus\nu(\eta)$ and $S^3\setminus\nu(J)$ are the homomorphisms coming from $\psi_K$ and $\psi_J$, respectively.

In summary, there is a bijection:
$$\begin{matrix}
	\Hom(\pi_1(M_{P(J)}),\Gamma)\\
	\updownarrow\\
	\{(\psi_K:\pi_1(M_K)\to\Gamma,\psi_J:\pi_1(M_J)\to\Gamma)|\psi_K(\eta)=\psi_J(\mu_J)\}
\end{matrix}.$$

The following is Lemma 2.1 in \cite{cochranDerivativesKnotsSecondorder2010}:

\begin{lemma}
\label{RhoSatelliteFormula}
In the notation in the previous paragraph,
	$$\rho(P(J),\varphi)=\rho(K,\varphi_K)+\rho(J,\varphi_J).$$
\end{lemma}

Recall the setup from Section \ref{Intro}: we have satellite operators $P_0$ and $P_1$ given by $K_0\sqcup\eta_0$ and $K_1\sqcup\eta_1$, respectively, which are homotopically related. That is, $K_0$ and $K_1$ are concordant, and $\eta_0\times\{0\}$ and $\eta_1\times\{1\}$ are homotopic through the complement of a concordance between $K_0$ and $K_1$. We showed in Theorem \ref{1solvable} that for any $J$, the difference $P_0(J)\#-P_1(J)$ is $(1)$-solvable.

\begin{prop}
\label{RhoInvObstrFails}
The set
	$$\{\rho(P_0(J)\#-P_1(J),\varphi_P) | P \text{ is a metabolizer for } \cA_0(P_0(J)\#-P_1(J))\}$$
contains 0. In other words, Theorem \ref{RhoInvObstr} fails to obstruct the $(1.5)$-solvability of the knot $P_0(J)\#-P_1(J)$.
\end{prop}

Before beginning the proof of Proposition \ref{RhoInvObstrFails}, we will fix some notation which will make it easier to work with metabelian groups.

\begin{definition}
The {\it metabelianization} of a group $G$ is defined to be $G/G^{(2)}$ and will be denoted $G^\metab$.
\end{definition}

Since homomorphisms to a metabelian group factor through the metabelianization of the domain, we can work with homomorphisms from metabelianizations. Since the longitude of a knot $K$ lies in the second commutator subgroup of the knot group, we have $\pi_1(K)^\metab\cong \pi_1(M_K)^\metab$, and so we can replace $\pi_1(M_K)$ with $\pi_1(K)$ or $\pi_1(K)^\metab$ in the discussion above. It will be easier to think in terms of knot groups than in terms of fundamental groups of 0-surgeries, and this is what we will do for the remainder of this section.

We have the following group theoretic lemma:

\begin{lemma}
\label{MetabelianLemma}
There is a canonical isomorphism $(G/G^{(2)})^{(1)} \cong G^{(1)}/G^{(2)}$
\end{lemma}

This lemma gives us the following corollary since the Alexander module of a knot group is the abelianization of its commutator subgroup:

\begin{cor}
The Alexander module of a knot is the commutator subgroup of the metabelianization of the knot group.
\end{cor}

\begin{proof}[Proof of Lemma \ref{MetabelianLemma}]
Throughout, we will denote the equivalence class of an element $x \in G \text{ or } G^{(1)}$ in $G/G^{(2)}$ or $G^{(1)}/G^{(2)}$ by $\bar{x}$.

Define a map $(G/G^{(2)})^{(1)} \to G^{(1)}/G^{(2)}$ by taking a commutator $[\bar{a},\bar{b}]$ of elements $\bar{a},\bar{b} \in G/G^{(2)}$ with representatives $a,b \in G$ to $\overline{[a,b]} \in G^{(1)}/G^{(2)}$. If we had chosen different representatives $ag,bg \in G$ for $\bar{a}$ and $\bar{b}$ (where $g \in G^{(2)}$), we have $\overline{[ag,bg]}=\overline{[a,b]}$ since we are working mod $G^{(2)}$, and so this map is well-defined.

Conversely, we can define a map $G^{(1)}/G^{(2)} \to (G/G^{(2)})^{(1)}$ by taking an element $\overline{[a,b]} \in G^{(1)}/G^{(2)}$ where $a,b \in G$ to the commutator $[\bar{a},\bar{b}]$. If we had chosen a different representative $[a,b]g$ for $\overline{[a,b]}$ (where $g \in G^{(2)}$), we see that its image is given by $[\bar{a},\bar{b}]\prod[\bar{x_i},\bar{y_i}]$ where $x_i,y_i \in G^{(1)}$ since $g \in G^{(2)}$. But, in $G/G^{(2)}$, commutators commute, and so this is equal to $[\bar{a},\bar{b}]$, and thus the map is well-defined.

These two maps are inverses, and so $(G/G^{(2)})^{(1)} \cong G^{(1)}/G^{(2)}$.
\end{proof}

We will also frequently use the fact that any group homomorphism $\varphi : H \to G$ induces a homomorphism on the metabelianizations $\varphi^\metab : H^\metab \to G^\metab$ (since $\varphi(H^{(2)}) \subseteq G^{(2)}$). This appears most frequently when $H$ is a subgroup of $G$ and $\varphi$ is the inclusion map.

We will now recall some fundamental facts about the Alexander modules and Blanchfield forms of satellite knots following \cite{livingstonAbelianInvariantsSatellite1985} (see also \cite{seifertHOMOLOGYINVARIANTSKNOTS1950}).

\begin{theorem}[Theorem 2 in \cite{livingstonAbelianInvariantsSatellite1985}]
\label{BlanchfieldSatelliteProperties}
The Alexander matrix (ie a matrix giving a $\Lambda$-module presentation for the Alexander module) of a satellite knot corresponding to a Seifert surface obtained from Seifert surfaces for $K$ and $J$ is given by
	$$A_{P(J)}(t) = A_K(t) \oplus A_J(t^w)$$
where $\oplus$ denotes the block sum, and $w=\lk(K,\eta)$ as in Section \ref{CGinvariants}. In particular, the Alexander module is given by
	$$\cA(P(J))\cong\cA(K) \oplus w\cA(J)$$
where, as an abelian group, $w\cA(J)$ is the direct sum of $w$ copies of $\cA(J)$, but carries the $\Lambda$-module structure defined by
	$$t\cdot(a_1,a_2,...,a_w)=(t \cdot a_w, a_1,...,a_{w-1}).$$
Moreover, after choosing suitable bases, the matrix for the Blanchfield form has a similar decomposition:
	$$B_{P(J)}(t) = B_K(t) \oplus B_J(t^w).$$
\end{theorem}

Note that the $\Lambda$-module $w\cA(J)$ can be seen as the first homology of the cover corresponding to the multiple of the abelianization map $w\cdot\ab:\pi_1(J)\to\Z$ (taken with $\Z$ coefficients, and coming with a $\Z$ action giving a $\Lambda$-module structure). Note also that if we restrict to the $\Z[t^w]$-module structure, then $w\cA(J)$ is the direct sum of $w$ copies of $\cA(J)$, where $t^w$ acts as the original $t$-action. Also, the map $w\cA(J)\to\pi_1(J)^\metab$ given by the inclusion map on each ($\Z$-module) $\cA(J)$-summand makes the following diagram commute (all maps represented by solid arrows are inclusions or induced by inclusions):

\setlength\mathsurround{0pt}

$$\begin{tikzcd}
	\cA(P(J)) \arrow{r}
		& \pi_1(P(J))^\metab\\
	w\cA(J) \arrow{u} \arrow[r,dashed]
		& \pi_1(J)^\metab \arrow{u}
\end{tikzcd}.$$

\setlength\mathsurround{0.8pt}

Implicit in the proof in \cite{livingstonAbelianInvariantsSatellite1985} of Theorem \ref{BlanchfieldSatelliteProperties} is that we have the following commutative diagram:

\setlength\mathsurround{0pt}

\begin{equation}
\begin{tikzcd}
	0 \arrow{r}
		& \cA(K) \arrow{r}{i} \arrow{d}{i}
		& \pi_1(K)^\metab \arrow{r}{\ab} \arrow{d}
		& \Z \arrow{r} \arrow{d}{\id}
		& 0\\
	0 \arrow{r}
		& \cA(P(J)) \arrow{r}{i}
		& \pi_1(P(J))^\metab \arrow{r}{\ab}
		& \Z \arrow{r}
		& 0\\
	& w\cA(J) \arrow{u}{i} \arrow{r}
		& \pi_1(J)^\metab \arrow{u} \arrow{r}{w\cdot\ab}
		& \Z \arrow{u}{\id}
\end{tikzcd}
\tag{$*$}
\label{eq:SatCommDiag}
\end{equation}

\setlength\mathsurround{0.8pt}

with exact rows, where the maps labeled with ``$i$'' are inclusion maps, the maps labeled with ``$\ab$'' are abelianization maps, the vertical maps to $\pi_1(P(J))^\metab$ are induced by the inclusion of the appropriate submanifolds, and the map $w\cA(J)\to\pi_1(J)^\metab$ is the map described above.

In the special case of a connect sum, we have
	$$\cA(K\#J)=\cA(K)\oplus\cA(J)$$
and similarly for the Blanchfield form and the commutative diagram, setting $w=1$.

\begin{proof}[Proof of Proposition \ref{RhoInvObstrFails}]
We have the following isomorphisms of Alexander modules
\begin{align*}
	\cA(P_0(J)\#-P_1(J))
		& \cong \cA(K_0) \oplus w\cA(J) \oplus \cA(-K_1) \oplus w \cA(-J)\\
		& \cong \cA(K_0\#-K_1) \oplus w \cA(J \# -J)
\end{align*}
and the corresponding decompositions of the Blanchfield forms. There are metabolizers of $\cA(K_0\#-K_1)$ and $w\cA(J\#-J)$ coming from slice disks, whose direct sum is then a metabolizer for $\cA(P_0(J)\#-P_1(J))$. Let $\psi:\pi_1(P_0(J)\#-P_1(J))^\metab\to\Gamma$ be the quotient map killing the metabolizer as above.

We will now use several applications of the satellite formula for metabelian $\rho$ invariants, cutting up $S^3$ along the three satellite tori $T_\#$, $T_0$, and $T_1$ shown in Figure \ref{fig:BaseptChoice2}. (Note that each torus is isotopic to the torus depicted in Figureref{fig:BaseptChoice} with respect to the corresponding satellite operations.) Notice that we will need to work with the fundamental group of the resulting submanifolds, and so want to choose an appropriate basepoint. So, the tori have been chosen to coincide at a particular ``spot,'' and we choose the basepoint to be in this spot.

\begin{figure}
	\centering
	\includegraphics[scale=0.7]{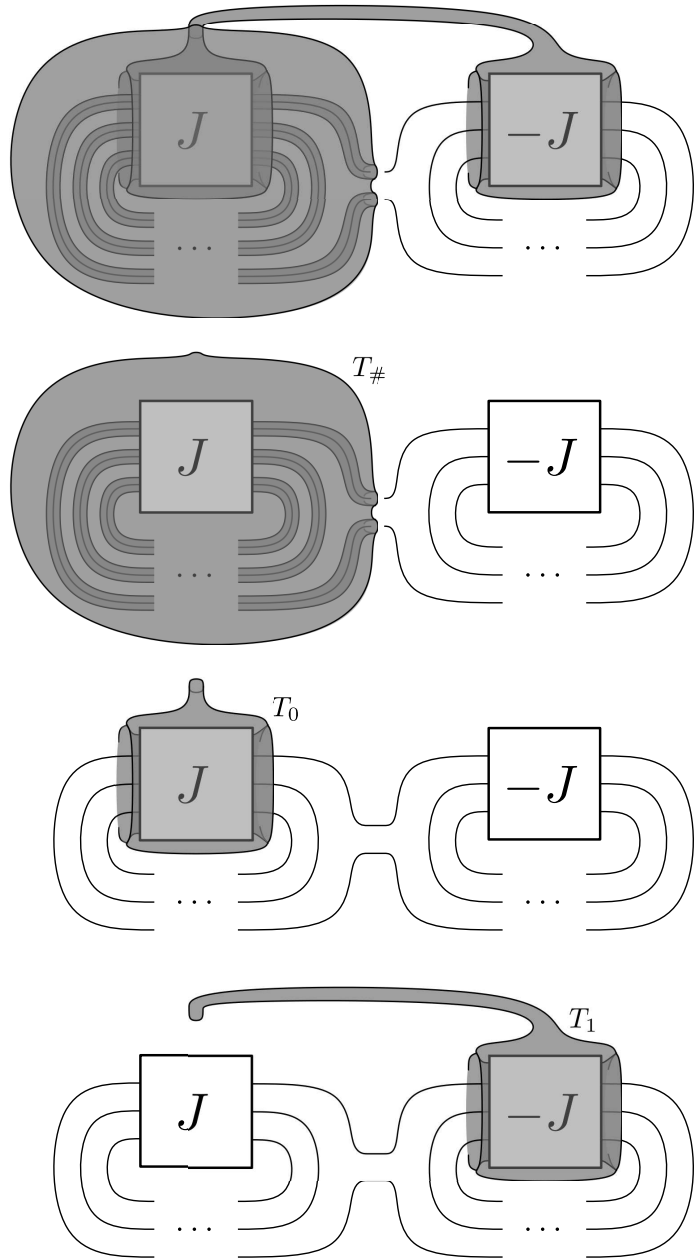}
	\caption{In the top diagram, $P_0(J)\#-P_1(J)$ is pictured with the three satellite tori $T_\#$, $T_0$, and $T_1$. In the bottom three diagrams, each of the tori is pictured individually. For the metabelian $\rho$ invariant calculations, take the basepoint to be in the spot common to all three tori.}
	\label{fig:BaseptChoice2}
\end{figure}

From the satellite formula for metabelian $\rho$ invariants, we have
	$$\rho(P_0(J)\#-P_1(J),\psi)=\rho(P_0(J),\psi_{P_0(J)})+\rho(-P_1(J),\psi_{-P_1(J)})$$
where $\psi_{P_0(J)}$ and $\psi_{-P_1(J)}$ are the homomorphisms obtained by restricting to the submanifolds which result from cutting $S^3$ along the ``swallow-follow'' torus $T_\#$ in Figure \ref{fig:BaseptChoice2}. We have the relation
	$$\psi_{P_0(J)}(\mu_{P_0(J)})=\psi_{-P_1(J)}(\mu_{-P_1(J)})$$
(since these are the same curve on $T_\#$).

After filling in the boundary components as we did in the discussion preceding Lemma \ref{RhoSatelliteFormula}, we can cut along the other two satellite tori $T_\epsilon$, $\epsilon\in\{0,1\}$ and use the satellite formula for metabelian $\rho$ invariants again to obtain:
	$$\rho(P_0(J)\#-P_1(J),\psi)=\rho(K_0,\psi_{K_0})+\rho(J,\psi_J)+\rho(-K_1,\psi_{-K_1})+\rho(-J,\psi_{-J})$$
with $\psi_{K_0}$, $\psi_{-K_1}$, $\psi_J$, and $\psi_{-J}$ coming from restrictions to submanifolds as before. From our gluings we must have the following identities:
\begin{align*}
	\psi_{K_0}(\eta_0)
		& = \psi_J(\mu_J)\\
	\psi_{K_0}(\mu_{K_0})
		& = \psi_{-K_1}(\mu_{K_1})\\
	\psi_{-K_1}(-\eta_1)
		& = \psi_{-J}(\mu_{-J})
\end{align*}
Also, note that $\mu_{K_0}$ and $\mu_{-K_1}$ are the same curves as $\mu_{P_0(J)}$ and $\mu_{-P_1(J)}$ from before, and the middle relation follows from the relation coming from the first application of the satellite formula.

Now, since $P_0$ and $P_1$ have the same winding numbers, the concatenation $\eta_0\cdot(-\eta_1)$ is nullhomologous in the complement of $K_0\#-K_1$, and so lies in $\cA(K_0\#-K_1)$ (recall that $-\eta_1$ denotes the mirror image of $\eta_1$ inside the complement of the mirror image of $K_1$ in the connect sum $K_0\#-K_1$). Moreover, by assumption $\eta_0$ and $\eta_1$ cobound an immersed annulus in the complement of the slice disk for $K_0\#-K_1$ and so $\eta_0\cdot(-\eta_1)$ is nullhomotopic in the complement of the slice disk, and so lies in the kernel of the inclusion induced map on $H_1(-,\Q[t^{\pm 1}])$. Therefore $\psi(\eta_0\cdot(-\eta_1))=1$, and so
	$$\psi_{K_0}(\eta_0)=\psi(\eta_0)=\psi(-\eta_1)=\psi_{-K_1}(-\eta_1).$$
Thus, since $\psi_{K_0}(\eta_0)=\psi_J(\mu_J)$ and $\psi_{-K_1}(-\eta_1)=\psi_{-J}(\mu_{-J})$, we must have
	$$\psi_J(\mu_J)=\psi_{-J}(\mu_{-J}).$$
Therefore there are maps $\psi_{K_0\#-K_1}$ and $\psi_{J\#-J}$ on $\pi_1(K_0\#-K_1)^\metab$ and $\pi_1(J\#-J)^\metab$, respectively, whose restrictions to the appropriate submanifolds are $\psi_{K_0}$, $\psi_{-K_1}$, $\psi_J$, and $\psi_{-J}$, and so we have:
\begin{align*}
	\rho(P_0(J)\#-P_1(J))
		& = \rho(K_0,\psi_{K_0})+\rho(J,\psi_J)+\rho(-K_1,\psi_{-K_1})+\rho(-J,\psi_{-J})\\
		& = \rho(K_0\#-K_1,\psi_{K_0\#-K_1})+\rho(J\#-J,\psi_{J\#-J})
\end{align*}

Putting together the commutative diagrams corresponding to the various satellite operations as in (\ref{eq:SatCommDiag}), we obtain the commutative diagram shown in Figure \ref{fig:RhoDiagram} (we have omitted the maps $\psi_{K_0}$, $\psi_{-K_1}$, $\psi_J$, and $\psi_{-J}$ to avoid cluttering up the diagram too much). The maps of (multiples of) Alexander modules are all inclusions of summands. This diagram commutes by the commutativity of the diagrams corresponding to the various satellite operations as in (\ref{eq:SatCommDiag}), plus the definitions of $\psi$, $\psi_{K_0\#-K_1}$, $\psi_{J\#-J}$, $\psi_{K_0}$, $\psi_{-K_1}$, $\psi_J$, and $\psi_{-J}$. In particular, the two pentagons with black arrows commute, and since the maps of (multiples of) Alexander modules are inclusions of summands, we see that the kernels of $\psi_{K_0\#-K_1}$ and $\psi_{J\#-J}$ are exactly the images of our chosen metabolizers of $\cA(K_0\#-K_1)$ and $w\cA(J\#-J)$, respectively.

\setlength\mathsurround{0pt}

\begin{figure}
$$\rotatebox{90}{\begin{tikzcd}[ampersand replacement = \&, column sep = small]
	\& \cA(K_0\#-K_1) \arrow[rrr]
			\arrow[dd]
		\&\&\& \pi_1(K_0\#-K_1)^\metab \arrow[rrdd, bend left = 45, dashed, "\psi_{K_0\#-K_1}"]\\
	\cA(K_0) \arrow[ru, gray]
			\arrow[rd, gray]
		\&\& \cA(-K_1) \arrow[lu, gray]
				\arrow[rrr, bend left = 15, gray]
				\arrow[ld, gray]
		\& \pi_1(K_0)^\metab \arrow[from=lll, bend right = 15, crossing over, gray]
				\arrow[ru, crossing over, gray]
				\arrow[rd, gray]
		\&\& \pi_1(-K_1)^\metab \arrow[lu, gray]
				\arrow[ld, gray]\\
	\& \cA(P_0(J)\#-P_1(J)) \arrow[rrr]
			\arrow[from=dd]
		\&\&\& \pi_1(P_0(J)\#-P_1(J))^\metab \arrow[rr, "\psi"]
		\&\& \Gamma\\
	w\cA(J) \arrow[ru, gray]
			\arrow[rd, gray]
		\&\& w\cA(-J) \arrow[lu, gray]
				\arrow[rrr, bend left = 15, gray]
				\arrow[ld, gray]
			\& \pi_1(J)^\metab \arrow[from=lll, bend right = 15, crossing over, gray]
				\arrow[ru, crossing over, gray]
				\arrow[rd, gray]
		\&\& \pi_1(-J)^\metab \arrow[lu, gray]
				\arrow[ld, gray]\\
	\& w\cA(J\#-J) \arrow[rrr]
		\&\&\& \pi_1(J\#-J)^\metab \arrow[rruu, bend right = 45, dashed, "\psi_{J\#-J}"]
\end{tikzcd}}$$
\caption{}
\label{fig:RhoDiagram}
\end{figure}

\setlength\mathsurround{0.8pt}

Since these metabolizers came from a slice disk, we have
\begin{align*}
	\rho(P_0(J)\#-P_1(J))
		& = \rho(K_0,\psi_{K_0})+\rho(J,\psi_J)+\rho(-K_1,\psi_{-K_1})+\rho(-J,\psi_{-J})\\
		& = \rho(K_0\#-K_1,\psi_{K_0\#-K_1})+\rho(J\#-J,\psi_{J\#-J})\\
		& = 0+0 = 0. 
\end{align*}
\end{proof}

\vfill
\pagebreak

%%%%%%%%%%%%%%%%%%%%   End of main body of article
%
%                             References
%
%   BiBTeX users uncomment the following line:
%
%\bibliographystyle{gtart}

\bibliographystyle{alpha} %All my references are undefined when I use the gtart file and I couldn't figure out how to fix it, so I figured I'd leave that to you.
\bibliography{RelatedSatelliteOperatorsBib}

\end{document}